\documentclass{article}
\usepackage[utf8x]{inputenc}
\usepackage{amssymb}
\usepackage{amsmath}
\usepackage{mathrsfs}
\usepackage{bm}
\usepackage{amsthm}
\usepackage{enumitem}
\usepackage{color}
\usepackage{graphicx}
\usepackage{bm}
\usepackage{ulem}
\usepackage{upgreek}

\usepackage{accents}

\newcommand{\ci}[1]{\mathscr{#1}}
\newcommand{\g}[1]{\mathfrak{#1}}
\renewcommand{\ni}{\nu}
\newcommand{\alfa}{\alpha}
\newcommand{\R}{\mathbf{R}}
\newcommand{\C}{\mathbf{C}}
\newcommand{\de}{\partial}
\renewcommand{\H}{\mathbf{H}}
\newcommand{\N}[1]{\left\lVert#1\right\rVert}
\newcommand{\e}{\varepsilon}

\renewcommand{\phi}{\varphi}
\newcommand{\con}[1]{\overline{#1}}

\newcommand{\mi}{\mu}
\renewcommand{\div}{\operatorname{div}}
\newcommand{\cerchio}[1]{\accentset{\smash{\raisebox{-0.12ex}{$\scriptstyle\circ$}}}{#1}\rule{0pt}{2.3ex}}

\newtheorem{proposizione}{Proposition}[section]
\newtheorem{teorema}[proposizione]{Theorem}
\newtheorem{lemma}[proposizione]{Lemma}
\newtheorem{corollario}[proposizione]{Corollary}

\theoremstyle{definition}
\newtheorem{definizione}[proposizione]{Definition}

\theoremstyle{remark}
\newtheorem{osservazione}[proposizione]{Remark}

\title{A compactness result for the CR Yamabe problem in three dimensions
}
\author{Claudio Afeltra}
\date{}

\begin{document}

\maketitle

\begin{abstract}
 We prove the compactness of the set of solutions to the CR Yamabe problem on a compact strictly pseudoconvex CR manifold of dimension three whose blow-up manifolds at every point have positive p-mass.
 As a corollary we deduce that compactness holds for CR-embeddable manifolds which are not CR-equivalent to $S^3$.
 The theorem is proved by blow-up analysis.
\end{abstract}

\section{Introduction}
After the Yamabe problem, that is the problem of finding a conformal metric with constant scalar curvature on a compact Riemannian manifold, was solved, the question of describing the set of solutions thereof naturally arose.

In negative or zero Yamabe class it is quite simple to prove that the solution is unique up to normalization, so the problem is non trivial only in the case of positive Yamabe class.
In a topics course at Stanford in 1988 (whose notes by Pollack are quite famous and easily found on the internet), Schoen conjectured that, fixing the value of the curvature, the set of solutions of the Yamabe problem
is compact except when the manifold is conformally equivalent to $S^n$, giving indications on how to prove it.

The conjecture was proved in dimension $n\le 24$ by Khuri, Marques and Schoen (see \cite{KMS} for the proof and a more detailed history of the problem), while it was proved to be false in dimension $n\ge 25$ (see \cite{B} and \cite{BM}). The theorem has also been proved in general dimension for conformally flat manifolds (see \cite{S}), or under the hypothesis that the Weyl tensor $W$ satisfies $|W|+|\nabla W|>0$ (see \cite{LZha}).

\vspace{3mm}

It is natural to study the same question on CR manifolds.
A CR manifold is a $2n+1$-dimensional manifold (with $2n+1\ge 3$) endowed with
an $n$-dimensional subbundle $\ci{H}\subset TM\otimes\C$ such that $[\Gamma(\ci{H}),\Gamma(\ci{H})]\subset\Gamma(\ci{H})$ and $\ci{H}\cap\con{\ci{H}}=\{0\}$.
Such kind of manifolds arise in the study of hypersurfaces of $\C^{n+1}$.
Under a generic hypothesis of nondegeneracy, CR manifolds have a canonical contact structure; a CR manifold with a contact form on this contact structure is called pseudohermitian manifold.

The choice of a contact form induces a rich geometric structure, including a pseudoriemannian metric (which is Riemannian for ``strictly pseudoconvex'' manifolds) and a connection called Tanaka-Webster connection. As in Riemannian geometry, the Tanaka-Webster connection allows to define a curvature tensor, and contracting it twice using the metric one gets a scalar quantity known as Webster curvature.

Since the choice of a contact form is determined up to the multiplication by a smooth function, problems from conformal geometry naturally extend to CR geometry, and trying to apply techniques therefrom to this setting showed that strong analogies hold.

In particular the problem of finding a contact form for compact pseudoconvex CR manifolds, known as CR Yamabe problem, has a variational formulation very similar to the one for the Yamabe problem: a contact form for the CR manifold $M$ has constant Webster curvature if and only if it is stationary for the CR Yamabe functional
$$\g{Q}(\theta) = \frac{\int_M R \theta\wedge{(d\theta)}^n}{\left(\int_M\theta\wedge{(d\theta)}^n\right)^{\frac{n}{n+1}}}$$
where $R$ is the Webster curvature.
The infimum of $\g{Q}$, called CR Yamabe class, is indicated by $\ci{Y}(M)$
\footnote{$\g{Q}(\theta)$ and $\ci{Y}(M)$ depend not only by $\theta$ and $M$ but also by the CR structure; by abuse of notation sometimes we omit the latter}
, and it can be shown with the same proof of the Riemannian case that if $\ci{Y}(M)<\ci{Y}(S^{2n+1})$ then $\g{Q}$ has minimizers.

Fixing a reference contact form $\theta$ and calling $\widetilde{\theta}=u^{\frac{2}{n}}\theta$,
$$\g{Q}(\widetilde{\theta}) = \frac{\int_M \left( b_n|\nabla^Hu|^2 +Ru^2 \theta\wedge{(d\theta)}^n\right)}{\left(\int_Mu^{2+\frac{2}{n}}\theta\wedge{(d\theta)}^n\right)^{\frac{n}{n+1}}} = \frac{\int_M uL_{\theta}u}{\left(\int_Mu^{2+\frac{2}{n}}\theta\wedge{(d\theta)}^n\right)^{\frac{n}{n+1}}}$$
where $\nabla^H$ is the subriemannian gradient, $b_n=2+\frac{2}{n}$ and $L_{\theta_0}=-b_n\Delta_b+R$, $\Delta_b$ being the sublaplacian. In particular the problem is equivalent to solving the equation $L_{\theta}u=u^{1+\frac{2}{n}}$.

Jerison and Lee proved that $\g{Q}$ has minimizers, thus solving the CR Yamabe problem, for non spherical (that is, not locally equivalent to $S^{2n+1}$) compact manifolds of dimension $2n+1\ge 5$ using the strategy that had been used for non conformally flat Riemannian manifolds of dimension $n\ge 6$ (see \cite{JL1}, \cite{JL3}).
The remaining cases of the problem were solved in \cite{G} and \cite{GY} using the method of critical points at infinity; the solutions found in this way are not necessarily minimizers of $\g{Q}$.

\vspace{3mm}

In the study of the problem of compactness for solutions to the Yamabe problem, two fundamental elements which do not have an obvious analogue in CR geometry are used: the positive mass theorem and the classification of metrics of constant positive scalar curvature on $\R^n$ conformal to the Euclidean metric.

The theory of mass for asymptotically flat pseudohermitian manifolds has yet to reach the extent of the analogous Riemannian theory; the techniques used in the study of the latter, in particular the theory of minimal hypersurfaces, do not yet have been developed to the required generality.

A positive mass theorem for spherical CR manifolds of dimension $2n+1\ge 7$ (and for $2n+1=5$ under a technical condition) was proved in \cite{CCY}.
A definition of pseudohermitian mass, or p-mass, valid for general three dimensional CR manifolds (but straightforwardly generalizable to higher dimension) has been given in \cite{CMY1}, where a CR positive mass theorem under certain hypotheses was proved.
In \cite{CC} a CR positive mass theorem was proved for spin five dimensional manifolds was proved.
A remarkable counterexample by Cheng, Malchiodi and Yang (see \cite{CMY2}) shows that, unlike in Riemannian geometry, the p-mass can be negative for three dimensional CR manifolds.
In analogy with other phenomena in CR geometry, many believe that such a counterexample can exist only in three dimension, and conjecture that the CR positive mass theorem be true in higher dimensions.

Like in Riemannian geometry, these theorems have applications to the study of the CR Yamabe functional $\g{Q}$: the aforementioned CR positive mass theorems imply as corollaries that under the respective hypotheses $\g{Q}$ has minimum, while the authors in \cite{CMY2} proved that in their counterexample thereto $\ci{Y}(M)$ is not attained.

Regarding metrics of constant positive scalar curvature on $\R^n$ conformal to the Euclidean metric, it was proved that they are the one obtained by pulling back the standard metric of $S^n$ through the stereographic projection, and translations and dilations thereof (Corollary 8.2 in \cite{CGS}). This result is used in a fundamental way when rescaling solutions tending to infinity in order to describe how they blow up.
In order to carry out the same strategy in CR geometry, an analogous theorem is needed for the Heisenberg group $\H^n$, which is the space obtained blowing up a pseudohermitian manifold at a point.
As in the Euclidean case, a family of solutions can be obtained by pulling back the standard contact form of $S^{2n+1}$ through the Cayley transform, which is a map $\H^n\to S^{2n+1}$ analogous to the stereographic projection, and it is natural to conjecture that they
are all the solutions. The problem is equivalent to the classification of the positive solutions of the equation $-\Delta_b u=u^{1+\frac{2}{n}}$.
The proof by Caffarelli, Gidas and Spruck uses the method of moving planes, but in $\H^n$ the symmetry group is not big enough to apply this method.

In \cite{JL2} Jerison and Lee classified the solutions of the CR Yamabe problem on $S^{2n+1}$, and through a removable singularity theorem, they proved the classification theorem under the hypothesis that $u\in L^{2+\frac{2}{n}}$, which is equivalent to having finite volume, but this result is not strong enough to carry out blow-up analysis, which requires the theorem at least for bounded solutions.

Very recently, Catino, Li, Monticelli and Roncoroni proved the classification theorem for $\H^1$ (see \cite{CLMR}).
This allows us to perform the blow-up analysis and to prove a compactness theorem for the solutions of the CR Yamabe problem.

Given a compact strictly pseudoconvex three-dimensional pseudohermitian manifold with $\ci{Y}(M,\ci{H})>0$ and $x\in M$ we will denote by $m_x$ the p-mass, as defined in Definition 2.3 in \cite{CMY1}, of the manifold $(M\setminus\{x\},\ci{H})$ with the blow-up contact form as defined in Subsection 2.1 in the same work.

Then we will prove the following theorem. As customary, due to the importance of subcritical approximation, we will prove a result slightly more general than the compactness of the solutions of the CR Yamabe problem, and which comprises subcritical exponents as well.

\begin{teorema}\label{IlTeorema}
 Let $(M,\ci{H},\theta)$ be a three-dimensional strictly pseudoconvex pseudohermitian manifold of positive CR Yamabe class
 such that, for every $x\in M$,
 $m_x>0$.
 Then for every $\e>0$ and $k\in\mathbf{N}$ there exists a constant $C$ such that
 $$\begin{cases}
    \frac{1}{C}\le u\le C\\
    \N{u}_{\Gamma^{k,\alfa}}\le C
   \end{cases}$$
 for every $u\in\cup_{1+\e\le p\le 3}\ci{M}_p$ and $0<\alfa<1$, where
 $$\ci{M}_p = \left\{ u>0 \;\middle|\; L_{\theta}u = u^p \right\},$$
 $L_{\theta}=-4\Delta_b+R$ being the conformal sublaplacian, and $\Gamma^{2,\alfa}$ is the Hölder space defined in Subsection \ref{CapitoloSpazi}.
 
 In particular, $\cup_{1+\e\le p\le 3}\ci{M}_p$ is compact in the $\Gamma^{k,\alfa}$ topology.
\end{teorema}

We point out that by known results in functional analysis (Theorem \ref{TeoremaRegolarita} below), in the theorem above the pseudohermitian Hölder spaces spaces $\Gamma^{k,\alfa}$ could be replaced by the standard Hölder spaces $C^{k,\alfa}$.

Furthermore, we remark that when $\ci{Y}(M)\le 0$ it is easy to prove uniqueness up to normalization (Theorem 7.1 in \cite{JL1}), so the problem is trivial.

\vspace{3mm}

The hypothesis in the theorem above is quite hard to check directly, since the computation of the p-mass can be rather complex even for simple CR manifolds (see \cite{CMY2}). Fortunately the aforementioned positive mass theorem of Cheng, Malchiodi and Yang (Theorem 1.1 in \cite{CMY1}), along with a later result of Takeuchi which allows to replace the hypothesis thereof with a simpler one (Theorem 1.1 in \cite{T}), allow us to state the following corollary to Theorem \ref{IlTeorema}.
We recall that a CR manifold is said to be embeddable in $\C^N$ whenever there exists a CR embedding in $\C^N$, that is an embedding $\Phi:M\to\C^N$ such that $d\Phi(\ci{H})\subseteq\operatorname{span}\left\{\frac{\de}{\de z_1},\ldots,\frac{\de}{\de z_N}\right\}$.

\begin{corollario}
 Let $(M,\ci{H},\theta)$ be a three-dimensional strictly pseudoconvex pseudohermitian manifold with $\ci{Y}(M,\ci{H})>0$ which is CR-embeddable in $\C^N$ for some $N$, and not CR-equivalent to $S^3$ with the standard CR structure.
 
 Then the thesis of Theorem \ref{IlTeorema} holds.
\end{corollario}

\begin{osservazione}
 Unlike the case of low-dimensional Riemannian manifolds, Theorem \ref{IlTeorema} is not true if the only hypothesis on $M$ is that it is not CR-equivalent to $S^3$. In fact by a result mentioned above (Theorem 1.3 in \cite{CMY2}) there exist three dimensional CR manifolds for which $\ci{Y}(M)$ is not attained.
 It is standard to prove that the subcritical functional
 $$\g{Q}_p(u) = \frac{\int_M uL_{\theta}u}{\left(\int_Mu^p\theta\wedge{(d\theta)}^n\right)^{\frac{1}{2}}}$$
 has a minimum point $u_p$ for $p\in[2,3)$ (Theorem 6.2 in \cite{JL1}). Were the thesis of Theorem \ref{IlTeorema} true, $u_{p_n}\to u$ in $C^2$ for some $p_n\to 3$, and by Lemma 6.4 (b) in \cite{JL1}, $u$ would be a minimum point for $\g{Q}$.
 
 We do not know whether there exists a CR manifold not CR-equivalent to $S^3$ for which the set of solution of the CR Yamabe problem $\left\{ u>0 \;\middle|\; L_{\theta}u = u^3 \right\}$ is not compact.
\end{osservazione}


\paragraph{Acknowledgements.} The author is supported by the fund “MIUR PRIN 2017” - CUP:E64I19001240001.
He would like to thank Andrea Malchiodi for suggesting the problem and for useful discussion.

\section{A review of CR geometry}\label{Preliminari}
We recall some facts about CR geometry that we will need in this work. For a complete introduction we refer to \cite{DT}.

A CR manifold is a manifold $M$ of dimension $2n+1$ endowed with a $n$-dimensional subbundle $\ci{H}$ of $TM\otimes\C$ such that $\ci{H}\cap\overline{\ci{H}}=\{0\}$ and $[\ci{H},\ci{H}]\subseteq\ci{H}$.

The Levi distribution is $H(M)=\g{Re}(\ci{H}\oplus\overline{\ci{H}})$.
A CR manifold is said nondegenerate if $H(M)$ is a contact distribution. In this article all CR manifolds will be assumed to be nondegenerate.
A contact form for $\H(M)$ is called a pseudohermitian form or pseudohermitian structure.
A nondegenerate CR manifold equipped with a contact form is called pseudohermitian manifold.

The sesquilinear form $L_{\theta}\in\Gamma(\ci{H}^*\otimes\overline{\ci{H}}^*)$ defined by
$$L_{\theta}(Z,\overline{W}) = i\theta([Z,\overline{W}]) = -id\theta(Z,\overline{W})$$
is called Levi form, and the degeneracy thereof is equivalent to the nondegeneracy of the CR structure.
A CR structure is said strictly pseudoconvex if it admits a contact form with positive definite Levi form. In such a case, all contact forms have definite Levi form, and $L_{\theta}$ is positive definite if and only if $L_{-\theta}$ is negative definite.
In a certain sense strictly pseudoconvex CR manifolds occupy, among CR manifolds, the role that Riemannian manifolds occupy among pseudoriemannian manifolds.
In the following we will assume all manifolds to be strictly pseudoconvex and all contact forms to have a positive Levi form.
Notice that when $n=1$, the main focus of this article, all manifolds are strictly pseudoconvex, and so up to a change of sign every contact form can be assumed to have positive Levi form.

The Levi form induces a symmetric form $G_{\theta}$ on $H(M)$
which is positive definite if and only if the pseudohermitian manifold is pseudoconvex.
In such a case $(M,H(M),G_{\theta})$ is a subriemannian manifold. We will call $d$ the Carnot-Carathéodory distance.

On a pseudohermitian manifold there exists a unique vector field $T$, called Reeb field, such that $\theta(T)=1$ and $i_T d\theta=0$.
It can be used to extend $G_{\theta}$ to a pseudoriemannian metric (Riemannian in the strictly pseudoconvex case) $g_{\theta}$, called the Webster metric, such that $T$ is orthogonal to $H(M)$ and $g_{\theta}(T,T)=1$.

Given a real function $u$ of class $C^1$, we define $\nabla^Hu$ as the subriemannian gradient, that is the section of $H(M)$ such that $G_{\theta}(\nabla^Hu, X)= Xu$ for every section $X$ of $H(M)$.

On a pseudohermitian manifolds there are two natural volume forms, the Riemannian one $dV$ associated to $g_{\theta}$ and the contact one $\theta\wedge(d\theta)^n$.
They just differ by a dimensional constant, and so choosing one over the other does not create substantial difference.
For commodity to apply the divergence theorem, in this article we will always use the first one.

Since we have a volume form, we can define the divergence of a vector field as
$$(\div X)dV = \ci{L}_X(dV).$$

We define the sublaplacian as the differential operator $\Delta_bu=\div(\nabla^Hu)$.
$\Delta_b$ is not elliptic, but it is degenerate elliptic and shares many properties with the Laplacian.


On any nondegenerate pseudohermitian manifold there exist a canonical connection called Tanaka-Webster connection.

To perform local computations, consider a local frame of $\ci{H}$, $(Z_1,\ldots,Z_n)$.
Define $Z_{\con{\alpha}}=\con{Z}_{\alpha}$ and $Z_0=T$. Let $(\theta^1,\ldots,\theta^n,\theta^{\con{1}},\ldots,\theta^{\con{n}},\theta^0)$ be the dual frame of $(Z_1,\ldots,Z_n,Z_{\con{1}},\ldots,Z_{\con{n}},Z_0)$ (in particular we take $\theta^0=\theta$). Given a tensor $\g{T}$ of type $(k,\ell)$ and indexes $A_i,B_i\in\{1,\ldots,n,\con{1},\ldots,\con{n},0\}$ let us define
$${\g{T}^{A_1\ldots A_k}}_{B_1\ldots B_{\ell}} = \g{T}(\theta^{A_1},\ldots,\theta^{A_k},Z_{B_1},\ldots,Z_{B_{\ell}}).$$
In the sequel Greek letters will denote indexes in $\{1,\ldots,n\}$.\\
Let $h_{\alpha\con{\beta}}=L_{\theta}(Z_{\alpha},Z_{\con{\beta}})$ be the coefficients of the Levi form. $h_{\alpha\con{\beta}}$ and its inverse $h^{\alpha\con{\beta}}$ can be used to raise and lower indexes. Therefore when using the Einstein convention, two repeated indexes are summed if one of them is barred and the other is non-barred when they are both up or both up, while in the other case they are summed if they are both barred or both non barred.

Since $\ci{H}$ is parallel, there exist 1-forms ${\omega_{\alpha}}^{\beta}$ such that
$$\nabla Z_{\alpha} = {\omega_{\alpha}}^{\beta}\otimes Z_{\beta}.$$
We separate indexes relative to covariant derivatives with respect to the Tanaka-Webster connection through a comma. Thus, for example, if $X=X_{\alfa}\theta^{\alfa}$ we have
$$X_{\alfa,\beta}= Z_{\beta}(X_{\alfa})-{\omega_{\alpha}}^{\gamma}(Z_{\beta})X_{\gamma}.$$

The following formula for the sublaplacian holds:
\begin{equation}\label{FormulaSublaplaciano}
 \Delta_b u = {u^{,\alpha}}_{\alpha}+ {u^{,\con{\alpha}}}_{\con{\alpha}}= u_{,\alpha\con{\alpha}}+u_{,\con{\alpha}\alpha}.
\end{equation}

The pseudo-Hermitian curvature tensor is the section of $\ci{H}^*\otimes\con{\ci{H}}^*\otimes\ci{H}^*\otimes\con{\ci{H}}$ defined as
$$R(Z,\con{W})X = \nabla_Z\nabla_{\con{W}}X - \nabla_{\con{W}}\nabla_ZX - \nabla_{[Z,\con{W}]}X$$
for $Z,W,X\in\Gamma(T^{(1,0)}M)$. We indicate the coordinates thereof by
$${R_{\alpha\con{\beta}\mi}}^{\ni} = L_{\theta}(R(Z_{\alpha},Z_{\con{\beta}})Z_{\mi},Z_{\ni}).$$
By contracting the pseudo-Hermitian curvature tensor, we get the pseudo-Hermitian Ricci tensor $R_{\alpha\con{\beta}}={R_{\mi\con{\beta}\alpha}}^{\mi}$, and contracting another time we get the Webster scalar curvature $R=h^{\alpha\con{\beta}}R_{\alpha\con{\beta}}$.

The conformal sublaplacian $L_{\theta}=-b_n\Delta_b + R$, where $b_n=2+\frac{2}{n}$, verifies the covariant transformation law
\begin{equation}\label{LeggeTrasformazioneSublaplaciano}
 L_{\widetilde{\theta}}\left(\frac{\phi}{u}\right) = u^{-\frac{n+2}{n}}L_{\theta}\phi
\end{equation}
where  $\widetilde{\theta}=u^{\frac{2}{n}}\theta$. In particular we obtain the formula for the transformation of the Webster curvature
\begin{equation}\label{LeggeTrasformazioneCurvatura}
 \widetilde{R}= u^{-\frac{n+2}{n}}L_{\theta}u =u^{-\frac{n+2}{n}}(-b_n\Delta_b + R)u.
\end{equation}


\subsection{The Heisenberg group}
The most important CR manifold is the Heisenberg group $\H^n$, that is $\C^n\times\R$ with the group law $(z,t)\cdot(w,s) = (z+w,t+s+2\g{Im}(z\cdot\con{w}))$, endowed with the left-invariant CR structure generated by the left-invariant vector fields
$$Z_{\alfa}= \frac{\de}{\de z_{\alfa}} +i\overline{z}_{\alfa}\frac{\de}{\de t}.$$
Up to a constant there exist a unique left-invariant 1-form annihilating $H(\H^n)$, that is $\theta= dt+i\sum_{\alfa}\left( z^{\alpha}d\con{z}^{\alpha}-\con{z}^{\alpha}dz^{\alpha}\right)$, and with this choice it can be easily computed that $\H^n$ is strictly pseudoconvex, and in particular that with respect to this frame $h_{\alfa\con{\beta}}=2\delta_{\alfa\con{\beta}}$.
Furthermore it can be proved that the Reeb vector field is $T=\frac{\de}{\de t}$ and that the Webster curvature is zero.

The Korányi norm is defined as $|(z,t)|=\left(|z|^4+t^2\right)^{\frac{1}{4}}$.

The fundamental solution for the CR sublaplacian of $\H^n$, which by formula \eqref{FormulaSublaplaciano} is
$$\Delta_b = \frac{1}{2}\sum_{\alfa=1}^n(Z_{\alfa}Z_{\con{\alfa}}+Z_{\con{\alfa}}Z_{\alfa}),$$
has been computed by Folland and Stein (see Section 6 in \cite{FS} or Theorem 3.9 in \cite{DT}): it holds that
\begin{equation}\label{SoluzioneFondamentale}
 -\Delta_b\left(\frac{\gamma_n}{|(z,t)|^{2n}}\right) =\delta_0
\end{equation}
with $\gamma_n = 2^{2n-2}\pi^{-n-1}\left|\Gamma\left(\frac{n}{2}\right)\right|^2$.

$\H^n$ admits a family of CR and group automorphisms analogous to the dilations of $\R^n$, defined by $\delta_{\lambda}(z,t)=(\lambda z,\lambda^2 t)$, and called Heisenberg dilations.

Since $(\delta_{\lambda})^*dV = \lambda^{2n+2}dV$, the number $Q=2n+2$ is called homogeneous dimension of $\H^n$.

$\H^n$ plays a role among pseudoconvex pseudohermitian manifolds analogous to the role of $\R^n$ among Riemannian manifolds;
indeed
every pseudoconvex pseudohermitian manifold can locally be appoximated with $\H^n$, through coordinates analogous to the normal coordinates of Riemannian geometry known as pseudohermitian normal coordinates.

Given $x\in M$, pseudohermitian normal coordinates are given by a function $\Phi:\Omega\to M$ with $0\in\Omega$ and $\Phi(0)=x$ approximating the pseudohermitian structure of $M$ as much as possible (see \cite{JL3} and the appendix for precise statements).
When using pseudohermitian normal coordinates, we will generally omit $\Phi$ and, with abuse of notation, we will identify $\Omega$ and $\Phi(\Omega)$. We will indicate by a circle the pseudohermitian objects coming from $\H^n$ through normal coordinates ($\cerchio{Z}_{\alfa},\cerchio{\Delta}_b,d\cerchio{V},\ldots$). Balls will be meant with respect to the Korányi distance, not to the Carnot-Carathéodory distance $d$. Anyways it is known that on $\H^n$ (or on balls thereof) the two distances are equivalent.

\begin{osservazione}\label{OsservazioneRiscalamento}
 When working in pseudohermitian coordinate around some point $\overline{x}$, we will often use dilations and translations to rescale the equation. If $u$ solves
 $$L_{\theta}u = \widetilde{R}u^p$$
 in $\Omega$, $x\in\Omega$ and $\lambda>0$ then $\widetilde{u}=\lambda^{\frac{2}{p_i-1}}u\circ\delta_{\lambda}\circ L_x$ (where $L_{x_i}(x)=x_i^{-1}x$ is the left translation) solves
 $$L_{\widetilde{\theta}}\widetilde{u} = (\widetilde{R}\circ\delta_{\lambda}\circ L_x)\widetilde{u}^p$$
 on $\delta_{\frac{1}{\lambda}}(L_{x^{-1}}(\Omega))$, where $\widetilde{\theta}=\frac{1}{\lambda^2}(\delta_{\lambda}\circ L_x)^*\theta$ and $L_{\widetilde{\theta}}$ is the conformal sublaplacian associated to the pseudohermitian manifold $(\delta_{\frac{1}{\lambda}}(L_{x^{-1}}(\Omega)),\widetilde{\theta},(\delta_{\lambda}\circ L_x)^*\ci{H})$.
 Notice that as $x\to\overline{x}$ and $\lambda\to\infty$,
 \begin{equation}\label{ConvergenzaStrutture}
  \left(\delta_{\frac{1}{\lambda}}(L_{x^{-1}}(\Omega)),\widetilde{\theta},(\delta_{\lambda}\circ L_x)^*\ci{H}\right) \to (\H^n,\cerchio{\theta},\cerchio{\ci{H}})
 \end{equation}
 smoothly on compact sets.
\end{osservazione}

In conformal geometry there exist local coordinates adapted to the conformal structure obtained using normal coordinates with respect to a metric in the conformal class that is ``as flat as possible'' around the base point, known as ``conformal normal coordinates'' and introduced by Lee and Parker in \cite{LP}.
Analogously in CR geometry local coordinates adapted to the CR structure, obtained using pseudohermitian normal coordinates with respect to an appropriate contact form, were introduced by Jerison and Lee in \cite{JL3}. They are known as CR normal coordinates.

We recall the recent classification theorem for contact forms with constant positive Webster curvature in $\H^1$ proved by Catino, Li, Monticelli and Roncoroni in \cite{CLMR}.

\begin{teorema}\label{Classificazione}
 If $u$ is a positive solution to
 $$-\Delta_b u = u^3$$
 on $\H^1$, then $u = U\circ\delta_{\lambda}\circ L_{x_0}$ for some $\lambda\in(0,\infty)$ and $x_0\in\H^1$, where
 $$U(z,t)=\frac{c_1}{\left(t^2+(1+|z|^2)^2\right)^{\frac{1}{2}}}$$
 is the Jerison-Lee solution and $L_{x_0}(x)=x_0^{-1}\cdot x$ is the left translation.
\end{teorema}

\subsection{Function spaces and estimates for the sublaplacian}\label{CapitoloSpazi}


In order to study analytical problems on CR manifolds, function spaces adapted to their structure are needed.
We refer to Section 5 in \cite{JL1} or Subsection 3.4 in \cite{DT} for their definition.

The Folland-Stein spaces $S^{k,p}(M)$ are analogous to the Sobolev spaces, but using only horizontal derivatives (the ones in the directions of $H(M)$).


Analogously, using the Carnot-Carathéodory distance, Hölder spaces which consider only horizontal derivatives, denoted by $\Gamma^{k,\alfa}(M)$, can be defined.

Let $C^{k,\alfa}(M)$ be the standard Hölder space relative to the Riemannian metric $g_{\theta}$.

Then the following immersions and regularity results hold.
These results were firstly stated in Section 5 in \cite{JL1} noticing that the proofs for analogous results by Folland and Stein in \cite{FS} for the Kohn Laplacian $\square_b$ are valid also for $\Delta_b$.

\begin{teorema}\label{TeoremaRegolarita}
 Let $U$ be a relatively compact open subset of a pseudohermitian manifold $(M,\ci{H},\theta)$.
 Then, if $\alfa\in(0,1)$, $k,\ell\in\mathbf{N}$, $p\in(1,\infty)$:
 \begin{enumerate}[topsep=2pt,itemsep=2pt,partopsep=2pt, parsep=2pt]
  \item $\N{u}_{\Gamma^{k,\alfa}(U)} \le C \N{u}_{S^{\ell,p}(M)}$ if $\frac{1}{p}=\frac{\ell-k-\alfa}{2n+2}$;
  \item $C_1\N{u}_{C^{\lfloor k/2\rfloor,\lfloor\alfa/2\rfloor +\{k/2\}}(U)} \le  \N{u}_{\Gamma^{k,\alfa}(M)} \le C_2 \N{u}_{C^{k,\alfa}(M)}$ where $\{k\}$ is the fractional part of $k$;
  \item $\N{u}_{S^{k+2,p}(U)}\le C\left(\N{\Delta_b u}_{S^{k,p}(M)} + \N{u}_{S^{k,p}(M)}\right)$;
  \item $\N{u}_{\Gamma^{k+2,\alfa}(U)}\le C\left(\N{\Delta_b u}_{\Gamma^{k,\alfa}(M)} + \N{u}_{\Gamma^{k,\alfa}(M)}\right)$.
 \end{enumerate}
 Furthermore for any of the above estimates there exists a $m$ such that in a $C^m$ neighborhood of the pseudohermitian structure $(\ci{H},\theta)$, the immersion holds with uniform constant.
\end{teorema}

The uniformity of the constant is not stated in \cite{FS}, but follows from the proof.
As in the case of the Laplacian, one can take $M=U$ in the above estimates if some regularity condition is assumed on $U$, but we will not need this.

As a last remark, we point out that the weak maximum principle holds for operators of the form $-\Delta_b+f$ with $f\ge 0$ because $\Delta_b$ is degenerate elliptic; furthermore for operators of the form $-\Delta_b+f$ the Harnack inequality holds with a constant depending on the pseudohermitian data and on $\N{f}_{L^{\infty}}$, and it can be proved by following the proof for elliptic operators from Theorem 8.20 in \cite{GT} (see for example \cite{GuL}
).

\section{A Pohozaev identity
}
As in the Riemannian case, where the blow-up analysis requires using the Pohozaev identity for $\R^n$, in the CR case we will need a Pohozaev formula for the Heisenberg group.
Pohozaev-type formulas have already been studied on the Heisenberg group (see \cite{GaL}) and on more general Carnot groups (see \cite{GV}).
We recall the formula that we will need, and for utility of the reader we include the proof.

We warn the reader that in the literature generally the subriemannian metric used is $\frac{1}{4}G_{\theta}$ instead of $G_{\theta}$, and this could cause unsubstantial differences in the computations.

In this section we will use the standard real frame for $H(\H^n)$
$$X_k = \frac{\de}{\de x_k}+2y_k\frac{\de}{\de t}, \;\;\; Y_k= X_{n+k}=\frac{\de}{\de y_k} -2x_k\frac{\de}{\de t}.$$
With respect to this frame the subriemannian gradient is
$$\nabla^H f = \frac{1}{4}\sum_{k=1}^n(X_kfX_k + Y_kfY_k)$$
and the sublaplacian
$$\Delta_b = \frac{1}{4}\sum_{k=1}^n(X_k^2+Y_k^2).$$
In order to state the formula, we define $\Xi$ as the vector field whose flow is the semigroup $\phi_t= \delta_{\log t}$, given explicitely by 
$$\Xi = \sum_{k=1}^n\left(x_k\frac{\de}{\de x_k}+y_k\frac{\de}{\de y_k}\right) +2t\frac{\de}{\de t} = \sum_{k=1}^n\left(x_kX_k+y_kY_k\right) +2tT =$$
\begin{equation}\label{FormulaCampoXi}
 =\sum_{\alfa=1}^n\left(z_{\alfa}Z_{\alfa}+\con{z}_{\alfa}Z_{\con{\alfa}}\right) +2tT.
\end{equation}
In the following proofs we will use some easily verified properties of $\Xi$.

\begin{lemma}\label{LemmaXi}
\begin{enumerate}[topsep=2pt,itemsep=2pt,partopsep=2pt, parsep=2pt]
 \item $\div\Xi=Q$;
 \item $[X_k,\Xi]=X_k$, $[Y_k,\Xi]=Y_k$;
 \item $f$ is homogeneous of degree $\alfa$ with respect to the Heisenberg dilations if and only if $\Xi f=\alfa f$.
\end{enumerate}
\end{lemma}
Note that the third of the above properties corresponds to Euler's theorem on $\R^n$.

We will denote by $d\sigma$ the Riemannian volume form of a submanifold and by $\ni$ the Riemannian normal.

\begin{proposizione}
 If $\Omega\subset\H^n$ is a bounded domain with $C^1$ boundary and $u\in C^2(\overline{\Omega})$ verifies the equation
 \begin{equation}\label{EquazioneDimPohozaev}
  -\Delta_bu = f(x,u)
 \end{equation}
 where $f\in C^1(\H^n\times\R)$, then, defining $F(x,t)=\int_0^tf(x,s)ds$, the following formula holds:
 $$\int_{\Omega}\left(QF(x,u) -\frac{Q-2}{2}uf(x,u) +(\Xi_xF)(x,u)\right)dV =$$
 \begin{equation}\label{PohozaevGenerale}
  =\int_{\de\Omega}\left(\left(F(x,u)-\frac{1}{2}|\nabla^H u|^2\right)\Xi\cdot\ni +\Xi u\nabla^H u\cdot\ni +\frac{Q-2}{2}u\nabla^H u\cdot\ni \right)d\sigma.
 \end{equation}
\end{proposizione}

\begin{proof}
 Multiplying equation \eqref{EquazioneDimPohozaev} by $\Xi u$ and integrating we get
 $$-\int_{\Omega}\Xi u\Delta_bu = \int_{\Omega} f(x,u)\Xi u = \int_{\Omega} \left[\Xi(F(x,u)) -(\Xi_x F)(x,u)\right]=$$
 \begin{equation}\label{FormulaDimPohozaev1}
  = \int_{\de\Omega}F(x,u)\Xi\cdot\ni -Q\int_{\Omega}F(x,u) -\int_{\Omega}(\Xi_x F)(x,u)
 \end{equation}
 where we used that $\div\Xi=Q$.
 Integrating by parts the right-hand side we get
 \begin{equation}\label{FormulaDimPohozaev2}
  -\int_{\Omega}\Xi u\Delta_bu = -\int_{\de\Omega}\Xi u\nabla^Hu\cdot\ni +\int_{\Omega}\nabla^Hu(\Xi u).
 \end{equation}
 Integrating by parts the last term we get, using that $[X_k,\Xi]=X_k$,
 $$\int_{\Omega}\nabla^Hu(\Xi u) = \frac{1}{4}\sum_{k=1}^{2n}\int_{\Omega}X_ku\cdot X_k(\Xi u) = \frac{1}{4}\sum_{k=1}^{2n}\int_{\Omega} X_ku\cdot (\Xi X_k+X_k)u =$$
 $$= \frac{1}{4}\sum_{k=1}^{2n}\int_{\Omega}\frac{1}{2}\Xi|X_ku|^2+ \frac{1}{4}\sum_{k=1}^{2n}\int_{\Omega}|X_ku|^2  = \frac{1}{2}\int_{\Omega}\Xi|\nabla^Hu|^2+ \int_{\Omega}|\nabla^Hu|^2 =$$
 $$= \frac{1}{2}\int_{\de\Omega}|\nabla^Hu|^2\Xi\cdot\ni -\frac{Q}{2}\int_{\Omega}|\nabla^Hu|^2 + \int_{\Omega}|\nabla^Hu|^2= $$
 \begin{equation}\label{FormulaDimPohozaev3}
  =\frac{1}{2}\int_{\de\Omega}|\nabla^Hu|^2\Xi\cdot\ni +\frac{2-Q}{2}\int_{\Omega}|\nabla^Hu|^2.
 \end{equation}
 Multiplying equation \eqref{EquazioneDimPohozaev} by $u$ and integrating we get
 \begin{equation}\label{FormulaDimPohozaev4}
  \int_{\Omega}uf(x,u) = -\int_{\Omega}u\Delta_bu = -\int_{\de\Omega}u\nabla^Hu\cdot\ni + \int_{\Omega}|\nabla^Hu|^2.
 \end{equation}
 Putting together formulas \eqref{FormulaDimPohozaev1}, \eqref{FormulaDimPohozaev2}, \eqref{FormulaDimPohozaev3} and \eqref{FormulaDimPohozaev4} we get the thesis.
\end{proof}

Now, using pseudohermitian normal coordinates, we can deduce a formula valid in small neighborhoods of any point of a pseudohermitian manifold.

\begin{proposizione}\label{PohozaevProp}
 Let $M$ be a pseudohermitian manifold, $\overline{x}\in M$, and let $u$ be a solution of the equation
 $$-b_n\Delta_bu +Ru = \widetilde{R}u^p.$$
 Then in pseudohermitian normal coordinates around $\overline{x}$ the following formula holds
 $$\int_{B_r}\left(\frac{1}{b_n}\left(\frac{Q}{p+1}-\frac{Q-2}{2}\right)\widetilde{R}u^{p+1} -\frac{1}{b_n}Ru^2+ \frac{1}{b_n}\frac{1}{p+1}\Xi(\widetilde{R})u^{p+1}+\right.$$
 $$\left.-\frac{1}{b_n}\frac{1}{2}\Xi(R)u^2 -\left(\Xi u+\frac{Q-2}{2}u\right)(\Delta_bu(x)-\cerchio{\Delta}_bu)\right)d\cerchio{V}=$$
 $$=\int_{\de B_r}\left(\left(\frac{1}{b_n}\frac{1}{p+1}\widetilde{R}u^{p+1}-\frac{1}{b_n}\frac{1}{2}Ru^2  \right)\Xi\cdot\cerchio{\ni} +\ci{B}(x,u,\cerchio{\nabla}^Hu)\right)d\cerchio{\sigma}$$
 where
 $$\ci{B}(x,u,\cerchio{\nabla}^Hu)= \frac{Q-2}{2}u(x)\cerchio{\nabla}^Hu(x)\cdot\cerchio{\ni}_{B_{d(x,0)}}(x) -\frac{1}{2}|\cerchio{\nabla}^Hu(x)|^2 \Xi(x)\cdot\cerchio{\ni}_{B_{d(x,0)}}(x) +$$
 $$+\Xi u(x)\cerchio{\nabla}^Hu(x)\cdot\cerchio{\ni}_{B_{d(x,0)}}(x)$$
 and $B_r$ denotes the ball with respect to the Korányi norm.
\end{proposizione}

\begin{proof}
 $u$ verifies
 $$-\cerchio{\Delta}_bu = -\frac{1}{b_n}Ru + \frac{1}{b_n}\widetilde{R}u^p + (\Delta_bu-\cerchio{\Delta}_bu).$$
 Defining
 $$f(x,u)= -\frac{1}{b_n}R(x)u(x) + \frac{1}{b_n}\widetilde{R}(x)u^p + (\Delta_bu(x)-\cerchio{\Delta}_bu(x))$$
 we have
 $$F(x,u)= -\frac{1}{b_n}R(x)u(x)u + \frac{1}{b_n}\frac{1}{p+1}\widetilde{R}(x)u^{p+1} + (\Delta_bu(x)-\cerchio{\Delta}_bu(x))u,$$
 and thus, applying formula \eqref{PohozaevGenerale}, we get
 $$\int_{B_r}\left(-\frac{Q}{b_n}Ru^2 + \frac{Q}{b_n}\frac{1}{p+1}\widetilde{R}u^{p+1} + Q(\Delta_bu-\cerchio{\Delta}_bu)u +\right.$$
 $$+\frac{Q-2}{2}\frac{1}{b_n}Ru^2 -\frac{Q-2}{2}\frac{1}{b_n}\widetilde{R}u^{p+1} -\frac{Q-2}{2}(\Delta_bu(x)-\cerchio{\Delta}_bu)u +$$
 $$\left.-\frac{1}{b_n}\Xi(Ru)u + \frac{1}{b_n}\frac{1}{p+1}\Xi(\widetilde{R})u^{p+1} + u\Xi(\Delta_bu-\cerchio{\Delta}_bu)\right)=$$
 $$=\int_{\de B_r}\left(\left(-\frac{1}{b_n}Ru^2 + \frac{1}{b_n}\frac{1}{p+1}\widetilde{R}u^{p+1} + (\Delta_bu-\cerchio{\Delta}_bu)u -\frac{1}{2}|\nabla^Hu|^2\right)\Xi\cdot\cerchio{\ni}  +\right.$$
 $$\left.+\Xi u\nabla^Hu\cdot\cerchio{\ni} +\frac{Q-2}{2}u\nabla^Hu\cdot\cerchio{\ni} \right)=$$
 $$=\int_{B_r}\left(u\Xi(\Delta_bu-\cerchio{\Delta}_bu) + (\Delta_bu-\cerchio{\Delta}_bu)\Xi u +Q(\Delta_bu-\cerchio{\Delta}_bu)u\right)+$$
 $$+\int_{\de B_r}\left(\left(-\frac{1}{b_n}Ru^2 + \frac{1}{b_n}\frac{1}{p+1}\widetilde{R}u^{p+1} -\frac{1}{2}|\nabla^Hu|^2\right)\Xi\cdot\cerchio{\ni}  +\right.$$
 \begin{equation}\label{FormulaDimPohozaev5}
  \left.+\Xi u\nabla^Hu\cdot\cerchio{\ni} +\frac{Q-2}{2}u\nabla^Hu\cdot\cerchio{\ni} \right).
 \end{equation}
 With some computations we can get
 $$-\frac{1}{b_n}\int_{B_r}\Xi(Ru)u = -\frac{1}{b_n}\int_{B_r}\left(\Xi(R)u^2 + \frac{1}{2}R\Xi(u^2)\right)=$$
 $$=-\frac{1}{b_n}\int_{B_r}\Xi(R)u^2 - \frac{1}{b_n}\frac{1}{2}\int_{\de B_r}Ru^2\Xi\cdot\cerchio{\ni} +\frac{1}{b_n}\frac{1}{2}\int_{B_r}\Xi(R)u^2 +\frac{1}{b_n}\frac{Q}{2}\int_{B_r}Ru^2=$$
 $$=-\frac{1}{b_n}\frac{1}{2}\int_{B_r}\Xi(R)u^2 - \frac{1}{b_n}\frac{1}{2}\int_{\de B_r}Ru^2\Xi\cdot\cerchio{\ni} +\frac{1}{b_n}\frac{Q}{2}\int_{B_r}Ru^2.$$
 Thus, substituting in formula \eqref{FormulaDimPohozaev5}, we get
 $$\int_{B_r}\left(-\frac{1}{b_n}\frac{Q+2}{2}Ru^2 + \frac{1}{b_n}\left(\frac{Q}{p+1}-\frac{Q-2}{2}\right)\widetilde{R}u^{p+1} +\right.$$
 $$\left.-\left(\Xi u+\frac{Q-2}{2}u\right)(\Delta_bu(x)-\cerchio{\Delta}_bu)+ \frac{1}{b_n}\frac{1}{p+1}\Xi(\widetilde{R})u^{p+1}\right)+$$
 $$-\frac{1}{b_n}\frac{1}{2}\int_{B_r}\Xi(R)u^2 - \frac{1}{b_n}\frac{1}{2}\int_{\de B_r}Ru^2\Xi\cdot\cerchio{\ni} +\frac{1}{b_n}\frac{Q}{2}\int_{B_r}Ru^2=$$
 $$=\int_{\de B_r}\left(\left(-\frac{1}{b_n}Ru^2 + \frac{1}{b_n}\frac{1}{p+1}\widetilde{R}u^{p+1} -\frac{1}{2}|\nabla^Hu|^2\right)\Xi\cdot\cerchio{\ni}  +\right.$$
 $$\left.+\Xi u\nabla^Hu\cdot\cerchio{\ni} +\frac{Q-2}{2}u\nabla^Hu\cdot\cerchio{\ni} \right)$$
 that is
 $$\int_{B_r}\left(\frac{1}{b_n}\left(\frac{Q}{p+1}-\frac{Q-2}{2}\right)\widetilde{R}u^{p+1} -\frac{1}{b_n}Ru^2+\right.$$
 $$\left.-\left(\Xi u+\frac{Q-2}{2}u\right)(\Delta_bu(x)-\cerchio{\Delta}_bu)+ \frac{1}{b_n}\frac{1}{p+1}\Xi(\widetilde{R})u^{p+1} -\frac{1}{b_n}\frac{1}{2}\Xi(R)u^2\right)=$$
 $$=\int_{\de B_r}\left(\left(-\frac{1}{b_n}\frac{1}{2}Ru^2 + \frac{1}{b_n}\frac{1}{p+1}\widetilde{R}u^{p+1} -\frac{1}{2}|\nabla^Hu|^2\right)\Xi\cdot\cerchio{\ni}  +\right.$$
 $$\left.+\Xi u\nabla^Hu\cdot\cerchio{\ni} +\frac{Q-2}{2}u\nabla^Hu\cdot\cerchio{\ni} \right)$$
 which is the thesis.
\end{proof}

In the following we will need an estimate of the term involving $\ci{B}$.

\begin{proposizione}\label{LimiteB}
 There exists a constant $c_n>0$ such that if $u$ is a function defined on a neighborhood of $0$ in $\H^n$ such that $u(x)= \frac{\gamma_n}{|x|^{Q-2}} + A + \alfa(x)$ with $\alfa(x)=o(1)$ for $x\to 0$, where $\gamma_n$ is the constant from formula \eqref{SoluzioneFondamentale}, then
 $$\lim_{\sigma\to 0}\sigma^{Q-1}\int_{\de B_1(0)}\ci{B}(x,u\circ\delta_{\sigma},\nabla^H (u\circ\delta_{\sigma})) = -c_nA.$$
\end{proposizione}

\begin{proof}
 $$\sigma^{Q-1}\int_{\de B_1}\ci{B}(x,u\circ\delta_{\sigma},\nabla^H (u\circ\delta_{\sigma})) = \sigma^{Q-1}\int_{\de B_1}\left(\frac{Q-2}{2}u\circ\delta_{\sigma}\nabla^H (u\circ\delta_{\sigma})\cdot\ni+\right.$$
 $$\left.-\frac{1}{2}|\nabla^H (u\circ\delta_{\sigma})|^2 \Xi\cdot\ni +\Xi(u\circ\delta_{\sigma})\nabla^H(u\circ\delta_{\sigma})\cdot\ni\right)=$$
 $$= \sigma^{Q-1}\int_{\de B_1}\left[\frac{Q-2}{2}\left(\frac{\gamma_n}{\sigma^{Q-2}|x|^{Q-2}} + A\right)\cdot \left(\frac{1}{\sigma^{Q-1}}\nabla^H\frac{\gamma_n}{|x|^{Q-2}}\right)\cdot\ni+ \right.$$
 $$-\frac{1}{2}\left|\nabla^H\left(\frac{\gamma_n}{\sigma^{Q-2}|x|^{Q-2}}\right)\right|^2\Xi\cdot\ni+$$
 $$\left.+\left(-(Q-2)\frac{\gamma_n}{\sigma^{Q-2}|x|^{Q-2}}\right)\left(\frac{1}{\sigma^{Q-1}}\nabla^H\frac{\gamma_n}{|x|^{Q-2}}\right)\cdot\ni\right] +o(1)=$$
 $$= \frac{Q-2}{2}A\int_{\de B_1}\nabla^H\left(\frac{\gamma_n}{|x|^{Q-2}}\right)\cdot\ni + $$
 $$-\frac{1}{\sigma^{Q-2}}\left[\int_{\de B_1}\frac{Q-2}{2}\frac{\gamma_n}{|x|^{Q-2}}\nabla^H\left(\frac{\gamma_n}{|x|^{Q-2}}\right)\cdot\ni -\frac{1}{2}\int_{\de B_1}\left|\nabla^H\left(\frac{\gamma_n}{|x|^{Q-2}}\right)\right|^2\Xi\cdot\ni\right] +o(1),$$
 where we used that $\Xi(|x|^{-Q+2})=(-Q+2)|x|^{-Q+2}$, which is a consequence of point 3 of Lemma \ref{LemmaXi}.
 Applying the divergence theorem we get, for $r>1$,
 $$\left(\int_{\de B_r}-\int_{\de B_1}\right)\left((Q-2)\frac{\gamma_n}{|x|^{Q-2}}\nabla^H\left(\frac{\gamma_n}{|x|^{Q-2}}\right)\cdot\ni +\left|\nabla^H\left(\frac{\gamma_n}{|x|^{Q-2}}\right)\right|^2\Xi\cdot\ni\right)= $$
 $$= \int_{B_r\setminus B_1}\left( (Q-2)\left|\nabla^H\left(\frac{\gamma_n}{|x|^{Q-2}}\right)\right|^2 +(Q-2)\frac{\gamma_n}{|x|^{Q-2}}\Delta_b\left(\frac{\gamma_n}{|x|^{Q-2}}\right) +\right.$$
 $$\left. +Q\left|\nabla^H\left(\frac{\gamma_n}{|x|^{Q-2}}\right)\right|^2 + \Xi\left|\nabla^H\left(\frac{\gamma_n}{|x|^{Q-2}}\right)\right|^2\right)=$$
 $$= \int_{B_r\setminus B_1}\left( (2Q-2)\left|\nabla^H\left(\frac{\gamma_n}{|x|^{Q-2}}\right)\right|^2 + (-2(Q-1))\left|\nabla^H\left(\frac{\gamma_n}{|x|^{Q-2}}\right)\right|^2\right)= 0.$$
 where we used Lemma \ref{LemmaXi} and the fact that $\left|\nabla^H\left(\frac{\gamma_n}{|x|^{Q-2}}\right)\right|^2$ is homogeneous of degree $-2(Q-1)$.
 
 Since, if $r\to\infty$,
 $$\int_{\de B_r}\left((Q-2)\frac{\gamma_n}{|x|^{Q-2}}\nabla^H\left(\frac{\gamma_n}{|x|^{Q-2}}\right)\cdot\ni +\left|\nabla^H\left(\frac{\gamma_n}{|x|^{Q-2}}\right)\right|^2\Xi\cdot\ni\right)\to 0 $$
 then
 $$\int_{\de B_1}\left((Q-2)\frac{\gamma_n}{|x|^{Q-2}}\nabla^H\left(\frac{\gamma_n}{|x|^{Q-2}}\right)\cdot\ni +\left|\nabla^H\left(\frac{\gamma_n}{|x|^{Q-2}}\right)\right|^2\Xi\cdot\ni\right)= 0 .$$
 Therefore
 $$\sigma^{Q-1}\int_{\de B_1}\ci{B}(x,u\circ\delta_{\sigma},\nabla^H (u\circ\delta_{\sigma})) = \frac{Q-2}{2}A\int_{\de B_1}\nabla^H\left(\frac{\gamma_n}{|x|^{Q-2}}\right)\cdot\ni +o(1).$$
 Using that $\ni = \frac{\nabla_{g_{\theta}}|x|}{|\nabla_{g_{\theta}}|x|}$ and that for every $f$, $g$ one has $\nabla_{g_{\theta}}f\cdot\nabla^Hg= \nabla^Hf\cdot\nabla^Hg$ as an easy consequence of the definitions, it is easy to prove that
 $$c_n = -\frac{Q-2}{2}\int_{\de B_1}\nabla^H\left(\frac{\gamma_n}{|x|^{Q-2}}\right)\cdot\ni$$
 is strictly positive, and so the thesis is proved.
\end{proof}

\section{Blow-up analysis}
In this section we carry out the blow-up analysis necessary in order to prove Theorem \ref{IlTeorema}.
This analysis follows closely the analogous analysis for Riemannian manifolds, which has be performed in various works: we refer, for example, to \cite{L}, \cite{LZhu} and \cite{M}.

We start by introducing in the context of CR manifolds some definitions due to Schoen
about blow-up points in conformal geometry which proved themselves fundamental in the study of nonlinear equations.

Let $M$ be a three-dimensional CR manifold endowed with a pseudohermitian structure $\theta$, $p_i\to 3$ a sequence with $1<p_i\le 3$ for every $i$, and $u_i\in C^2(M)$ a sequence of positive solutions of
\begin{equation}\label{EquazioneCapitolo4}
 L_{\theta}u_i = \widetilde{R}u_i^{p_i},
\end{equation}
where $\widetilde{R}$ is a positive function of class $C^1$.

\begin{definizione}
 A point $\overline{x}\in M$ is called a blow-up point if there exist a sequence $x_i\to\overline{x}$ such that $M_i=u_i(x_i)\to\infty$.
\end{definizione}

\begin{definizione}
 A point $\overline{x}\in M$ is called an isolated blow-up point if there exist $\overline{r}>0$, a constant $C$, and a sequence $x_i\to\overline{x}$ such that $x_i$ is a local maximum of $u_i$, $u_i(x_i)\to\infty$, and
 $$u_i(x)\le Cd(x,x_i)^{-\frac{2}{p_i-1}}$$
 for every $x\in B_{\overline{r}}(x_i)$.
\end{definizione}

Given an isolated blow-up point $\overline{x}$, define
$$\overline{u}_i(r) = \int_{\de B_1(x_i)}u_i\circ\delta_rd\cerchio{\sigma}$$
in pseudohermitian normal coordinates, and $\overline{w}_i(r)= r^{\frac{2}{p_i-1}}\overline{u}_i(r)$.

\begin{definizione}\label{DefinizionePuntoIsolatoSemplice}
 An isolated blow-up point $\overline{x}$ is called an isolated simple blow-up point $\overline{x}$ if there exists $\rho\in(0,\overline{r})$ (independent of $i$) such that $\overline{w}_i$ has exactly one critical point in $(0,\rho)$.
\end{definizione}

Notice that taking the definition from conformal geometry in a straightforward manner would not be appropriate because in $\H^1$, unlike in $\R^n$, the family of surface measures on $\de B_r$ does not have homogeneity properties, so we have to take this slight variation suited to the blow-up procedure, which relies on dilations in a fundamental manner.

\begin{lemma}\label{Harnack}
 If $\overline{x}$ is an isolated blow-up point then there exists $C$ such that for $0<r<\frac{\overline{r}}{3}$ it holds that
 $$\max_{B_{2r}\setminus B_{r/2}}u_i \le C \min_{B_{2r}\setminus B_{r/2}}u_i.$$
\end{lemma}

\begin{proof}
 Let $w_i=r^{\frac{2}{p_i-1}}u_i\circ\delta_r$. If $\theta_r=\frac{1}{r^2}(\delta_r)^*\theta$, using the notation of Remark \ref{OsservazioneRiscalamento}, $w_i$ verifies
 $$L_{\theta_r}w_i = (\widetilde{R}\circ\delta_r)w_i^{p_i}$$
 on $\delta_{\frac{1}{r}}B_{\overline{r}}$, and the definition of isolated blow-up and the properties of pseudohermitian normal coordinates imply that $w_i(x)\le\frac{C}{|x|^{\frac{2}{p_i-1}}}$.
 So $w_i$ is uniformly bounded outside some neighborhood of $0$ with bound independent by $r$, and the thesis follows the Harnack inequality, which we can apply with constant independent by $L_{\theta_r}$ thanks to formula \eqref{ConvergenzaStrutture}.
\end{proof}

In the following, given an isolated blow-up point $x_i\to\overline{x}$, in order to study the blowing up sequence of functions we do a rescaling defining $M_i=u_i(x_i)$ and
$$v_i=\frac{1}{M_i}u_i\circ\delta_{M_i^{-\frac{p_i-1}{2}}}\circ L_{x_i}$$
defined on $B_{\overline{r}M_i^{\frac{p_i-1}{2}}}(\overline{x}$).
\footnote{By our notation this denotes the ball in the Korányi distance, while the definition of simple blow-up involves the Carnot-Carathéodory distance, but since the two are equivalent, up to changing $\overline{r}$ there is no substantial difference in using one or the other.}
Following the notation of Remark \ref{OsservazioneRiscalamento}, $v_i$ verifies
$$L_{\theta_i}v_i = (\widetilde{R}\circ\delta_{M_i^{-\frac{p_i-1}{2}}})v_i^{p_i}$$
where $L_{\theta_i}$ is the sublaplacian with respect to the rescaled contact form $\theta_i= M_i^{-(p_i-1)}\left(\delta_{M_i^{-\frac{p_i-1}{2}}}\circ L_{x_i}\right)^*\theta$ on the rescaled CR structure.

In the following all covariant derivatives applied to $v_i$ are meant with respect to this rescaled pseudohermitian structure.

\begin{proposizione}\label{BlowUpIsolatiBolle}
 If $\overline{x}$ is an isolated blow-up point then for any $R_i\to\infty$, $\e_i\to 0$ and $k\in\mathbf{N}$, up to subsequences, in pseudohermitian normal coordinates around $\overline{x}$ it holds that
 $$\N{\frac{1}{M_i}u_i\left(\delta_{M_i^{-\frac{p_i-1}{2}}}(x_i^{-1} \cdot x)\right) -(U\circ\delta_{\widetilde{R}(0)^{1/2}})(x)}_{\Gamma^{k,\alfa}(B_{R_i}(0))}\le\e_i,$$
 where $U$ is defined in Theorem \ref{Classificazione}, $M_i=u_i(x_i)$, and
 $$\frac{R_i}{\log M_i} \to 0.$$
\end{proposizione}

\begin{proof}
 
 Using the notation before the statement of this theorem, $v_i$ satisfies
 \begin{equation}\label{SistemaBlowUpIsolatiBolle}
  \begin{cases}
   L_{\theta_i}v_i = (\widetilde{R}\circ\delta_{M_i^{-\frac{p_i-1}{2}}})v_i^{p_i}\\
   v_i(0)=1\\
   \nabla^H_{\theta_i}v_i(0) = 0\\
   0< v_i(x)< Cd(x,0)^{-\frac{2}{p_i-1}}.
  \end{cases}
 \end{equation}
 Thanks to Lemma \ref{Harnack}, there exists $C$ such that
 $$\max_{\de B_r}v_i \le C\min_{\de B_r}v_i$$
 for $r<1$.
 Let $\psi>0$ be a function such that $\psi^2\theta$ has positive Webster curvature
 in some ball $B_{\widetilde{r}}(0)$ (for example $\phi=U\circ\delta_{\lambda}$ for $\lambda$ big enough)
 and define $\psi_i= \psi\circ\delta_{M_i^{-\frac{p_i-1}{2}}}$.
 The operator
 $L_{\psi_i^2\theta_i}$
 satisfies the maximum principle, and since, by formula \eqref{LeggeTrasformazioneSublaplaciano},
 $$L_{\psi_i^2\theta_i}\left(\psi_i^{-1}v_i\right) = \psi_i^{-3}L_{\theta_i}v_i >0,$$
 then
 for any $r\le\widetilde{r}$ we have
 $\min_{B_r}\psi_i^{-1}v_i \ge \min_{\de B_r}\psi_i^{-1}v_i$,
 that is
 $$\min_{B_r}v_i \ge C\min_{\de B_r}v_i$$
 for some $C>0$.
 Therefore
 if $r\le\widetilde{r}$,
 $$\max_{\de B_r}v_i \le C\min_{\de B_r}v_i \le C_1 \min_{B_r}v_i\le C_1v_i(0)=C_1 .$$
 Thanks to this and to the last inequality in equation \eqref{SistemaBlowUpIsolatiBolle} we get that $v_i\le C$ in $B_{\overline{r}M_i^{\frac{p_i-1}{2}}}$ for any $i$ and some $C$.
 
 
 Applying Theorem points 3 and 1 of Theorem \ref{TeoremaRegolarita},
 we can deduce that for any bounded open subset $V\subset\H^1$, eventually $\N{v_i}_{\Gamma^{1,\alfa}_{\theta_i}(V)}\le C_V$ for every $\alfa\in(0,1)$.
 By a standard bootstrap argument using point 4 of Theorem \ref{TeoremaRegolarita}, for any $k$ we have $\N{v_i}_{\Gamma^{k,\alfa}_{\theta_i}(V)}\le C_{V,k}$, and this implies, by point 2 of the same theorem, the uniformity statement thereof and formula \eqref{ConvergenzaStrutture}, that $\N{v_i}_{C^{k,\alfa}(V)}\le C'_{V,k}$ for every $k$.
 
 This implies that, for any $k$ and $R>0$, up to subsequences $v_i$ tends to some limit $v$ in $C^{k,\alfa}(B_R)$.
 By a diagonal argument, up to a further passage to subsequences we get that $v_i\to v$ in $C^{k,\alfa}(B_R)$ for any $k$ and $R$, with $v$ defined on $\H^1$ and satisfying:
 

 \begin{equation}
  \begin{cases}
   L_{\cerchio{\theta}}v = \widetilde{R}(0) v^3\\
   v(0)=1\\
   \cerchio{\nabla}^H v(0) = 0\\
   v>0.
  \end{cases}
 \end{equation}
 The thesis of the theorem follows from Catino-Li-Monticelli-Roncoroni's classification theorem (Theorem \ref{Classificazione}).
\end{proof}

We point out that in the proof of the theorem above we had to be slightly more careful than in the proof of the analogous Riemannian result in handling regularity theorems because, while Sobolev and Hölder spaces with respect to different Riemannian metrics on a compact manifold coincide, this is not true for the spaces $S^{k,\alfa}$ and $\Gamma^{k,\alfa}$ with respect to two different pseudohermitian structures $(\ci{H}_1,\theta_1)$ and $(\ci{H}_2,\theta_2)$ on the same manifold, because they are defined through derivatives along the respective Levi distributions $H_1(M)$ and $H_2(M)$, which do not necessarily coincide.

\begin{lemma}\label{LemmaStimaPuntiIsolatiSemplici}
 Let $\overline{x}$ an isolated simple blow-up point, $R_i\to\infty$, and suppose that the thesis of Lemma \ref{BlowUpIsolatiBolle} holds for some $\e_i\to 0$.
 Then, given a fixed, sufficiently small $\delta>0$, there exists $\rho_1\in(0,\rho)$ ($\rho$ being the one from Definition \ref{DefinizionePuntoIsolatoSemplice}) such that
 $$u_i(x)\le C M_i^{-\lambda_i}d(x,x_i)^{-2+\delta},$$
 $$|(u_i)_{,1}(x)|\le C M_i^{-\lambda_i}d(x,x_i)^{-3+\delta},$$
 $$|(u_i)_{,11}(x)|\le C M_i^{-\lambda_i}d(x,x_i)^{-4+\delta}$$
 $$|(u_i)_{,1\con{1}}(x)|\le C M_i^{-\lambda_i}d(x,x_i)^{-4+\delta},$$
 $$|(u_i)_{,0}(x)|\le C M_i^{-\lambda_i}d(x,x_i)^{-4+\delta}$$
 $$|(u_i)_{,01}(x)|\le C M_i^{-\lambda_i}d(x,x_i)^{-5+\delta},$$
 $$|(u_i)_{,00}(x)|\le C M_i^{-\lambda_i}d(x,x_i)^{-6+\delta},$$
 for $R_iM_i^{-\frac{p_i-1}{2}}\le d(x,x_i)\le\rho_1$,
 where $\lambda_i=(2-\delta)\frac{p_i-1}{2}-1$.
\end{lemma}

\begin{proof}
 The proof is analogous to the proof of Lemma 3.3 in \cite{LZhu}.
\end{proof}

\begin{lemma}\label{StimaSoluzioneRiscalata}
 In the hypotheses of Lemma \ref{LemmaStimaPuntiIsolatiSemplici}
 for $|x|\le\rho_1M_i^{\frac{p_i-1}{2}}$
 the following estimates hold:
 $$v_i(x)\le CM_i^{\frac{p_i-1}{2}\delta}(1+|x|)^{-2},$$
 $$|(v_i)_{,1}(x)|\le CM_i^{\frac{p_i-1}{2}\delta}(1+|x|)^{-3},$$
 $$|(v_i)_{,11}(x)|\le CM_i^{\frac{p_i-1}{2}\delta}(1+|x|)^{-4},$$
 $$|(v_i)_{,1\con{1}}(x)|\le CM_i^{\frac{p_i-1}{2}\delta}(1+|x|)^{-4},$$
 $$|(v_i)_{,0}(x)|\le CM_i^{\frac{p_i-1}{2}\delta}(1+|x|)^{-4},$$
 $$|(v_i)_{,01}(x)|\le CM_i^{\frac{p_i-1}{2}\delta}(1+|x|)^{-5},$$
 $$|(v_i)_{,00}(x)|\le CM_i^{\frac{p_i-1}{2}\delta}(1+|x|)^{-6}.$$
\end{lemma}

\begin{proof}
 It follows from Lemmas \ref{BlowUpIsolatiBolle}
 and \ref{LemmaStimaPuntiIsolatiSemplici}.
\end{proof}

\begin{lemma}\label{StimaSublaplacianoCoordinate}
 If $\overline{x}$ is an isolated blow-up point then, with the notation of Lemma \ref{LemmaStimaPuntiIsolatiSemplici}, in pseudohermitian normal coordinates around $x_i$
 $$\left|\int_{B_{\rho_1}(x_i)}\left(u_i +\Xi u_i\right)(\cerchio{\Delta}_bu_i -\Delta_b u_i)\right| \le C M_i^{-2+2\delta + o(1)}$$
\end{lemma}

\begin{proof}
 Thanks to formulas \eqref{FormulaCampoXi} and \eqref{RelazioniInverseCoordinatePH} and Lemma \ref{StimaSoluzioneRiscalata}, and the fact that $v_i$ is real, we have
 \footnote{Estimate \eqref{RelazioniInverseCoordinatePH} has been proved for a fixed pseudohermitian structure, while the functions $v_i$ are differentiated with respect to a sequence of pseudohermitian structures, but since that sequence converges in $C^{\infty}$, the estimate hold with an uniform constant. This fact is used also later.}
 $$\left|v_i +\Xi v_i\right|\lesssim v_i +|x||\cerchio{Z}v_i| + |x|^2|\cerchio{T}v_i| \lesssim v_i +|x||(v_i)_{,1}| + |x|^2|(v_i)_{,0}| \lesssim$$
 $$\lesssim M_i^{\frac{p_i-1}{2}\delta}(1+|x|)^{-2} + |x|M_i^{\frac{p_i-1}{2}\delta}(1+|x|)^{-3} + |x|^2M_i^{\frac{p_i-1}{2}\delta}(1+|x|)^{-4}\lesssim$$
 $$\lesssim M_i^{\frac{p_i-1}{2}\delta}(1+|x|)^{-2}.$$ 
 Furthermore define $v_i=\frac{1}{M_i}u_i\circ\delta_{M_i^{-\frac{p_i-1}{2}}}$.
 Since $u_i$ is real, using Lemmas \ref{DisuguaglianzaSublaplaciano} and \ref{StimaSoluzioneRiscalata}
 $$|\Delta_bu_i - \cerchio{\Delta}_bu_i| \lesssim |x||(u_i)_{,1}| + |x|^2|(u_i)_{,0}| + |x|^2|(u_i)_{,1\con{1}}| +|x|^2|(u_i)_{,11}| +$$
 $$+|x|^3|(u_i)_{01}| +|x|^6|(u_i)_{,00}| \lesssim$$
 $$\lesssim M_i\left(M_i^{-\frac{p_i-1}{2}}|x|M_i^{\frac{p_i-1}{2}}|(v_i)_{,1}| + M_i^{-(p_i-1)}|x|^2M_i^{p_i-1}|(v_i)_{,0}|+\right.$$
 $$+ M_i^{-(p_i-1)}|x|^2M_i^{p_i-1}|(v_i)_{,1\con{1}}| +M_i^{-(p_i-1)}|x|^2M_i^{p_i-1}|(v_i)_{,11}| +$$
 $$\left.+ M_i^{-3\frac{p_i-1}{2}}|x|^3M_i^{3\frac{p_i-1}{2}}|(v_i)_{01}| +M_i^{-3(p_i-1)}|x|^6M_i^{2(p_i-1)}|(v_i)_{,00}|\right)\circ\delta_{M_i^{\frac{p_i-1}{2}}}\lesssim$$
 $$\lesssim M_iM_i^{\frac{p_i-1}{2}\delta}\left(|x|(1+|x|)^{-3} + |x|^2(1+|x|)^{-4}+ |x|^2(1+|x|)^{-4} +\right.$$
 $$\left.+|x|^2(1+|x|)^{-4} + |x|^3(1+|x|)^{-5} +M_i^{-(p_i-1)}|x|^6(1+|x|)^{-6}\right)\circ\delta_{M_i^{\frac{p_i-1}{2}}}\lesssim$$
 $$\lesssim M_iM_i^{\frac{p_i-1}{2}\delta}\left(|x|(1+|x|)^{-3} +M_i^{-(p_i-1)}|x|^6(1+|x|)^{-6}\right)\circ\delta_{M_i^{\frac{p_i-1}{2}}}.$$
 So
 $$\left|\int_{B_{\rho_1}(x_i)}\left(u_i +\Xi u_i\right)(\cerchio{\Delta}_bu_i -\Delta_b u_i)\right| =$$
 $$= M_i M_i^{-2(p_i-1)}\left|\int_{B_{\rho_1M_i^{\frac{p_i-1}{2}}}(0)}\left(v_i +\Xi v_i\right)\cdot (\Delta_bu_i - \cerchio{\Delta}_bu_i)\circ\delta_{M_i^{-\frac{p_i-1}{2}}} \right|\lesssim$$
 $$\lesssim M_i^2 M_i^{-2(p_i-1)}M_i^{(p_i-1)\delta}\int_{B_{\rho_1M_i^{\frac{p_i-1}{2}}}(0)}(1+|x|)^{-2}\cdot\left[|x|(1+|x|)^{-3} +\right.$$
 $$\left. +M_i^{-(p_i-1)}|x|^6(1+|x|)^{-6}\right] \lesssim$$
 $$\lesssim M_i^{4-2p_i}M_i^{(p_i-1)\delta}\left[\log M_i + M_i^{-(p_i-1)}M_i^2\right] = M_i^{-2+2\delta + o(1)}.$$
\end{proof}

\begin{lemma}\label{StimaTau}
 In the hypotheses of Lemma \ref{LemmaStimaPuntiIsolatiSemplici}, if $\tau_i=\frac{Q+2}{Q-2}-p_1 = 3-p_i$ then
 $$\tau_i = O(u_i(x_i)^{-2 +2\delta +o(1)})$$
 and in particular $u_i(x_i)^{\tau_i}\to 1$.
\end{lemma}

\begin{proof}
 Let us apply the Pohozaev identity of Proposition \ref{PohozaevProp} with respect to the base point $x_i$ with $r=\rho_1$:
 $$\int_{B_{\rho_1}}\left(\left(\frac{1}{p_i+1}-\frac{1}{4}\right)\widetilde{R}u_i^{p_i+1} -\frac{1}{4}Ru_i^2+ \frac{1}{4}\frac{1}{p_i+1}\Xi(\widetilde{R})u_i^{p_i+1}+\right.$$
 $$\left.-\frac{1}{8}\Xi(R)u_i^2 -\left(\Xi u_i+u_i\right)(\Delta_bu(x)-\cerchio{\Delta}_bu)\right)d\cerchio{V}=$$
 $$=\int_{\de B_{\rho_1}}\left(\left(\frac{1}{4}\frac{1}{p_i+1}\widetilde{R}u_i^{p_i+1}-\frac{1}{8}Ru_i^2  \right)\Xi\cdot\cerchio{\ni} +\ci{B}(x,u_i,\cerchio{\nabla}^Hu)\right)d\cerchio{\sigma}.$$
 Let us estimate the various terms.
 $$\int_{B_{\rho_1}}\Xi(\widetilde{R})u_i^{p_i+1} \lesssim$$
 $$\lesssim M_i^{p_i+1}M_i^{-2(p_i-1)}\int_{B_{\rho_1M_i^{\frac{p_i-1}{2}}}}(M_i^{-\frac{p_i-1}{2}}|x|+M_i^{-(p_i-1)}|x|^2)v_i^{p_i+1} \lesssim$$
 $$\lesssim M_i^{3-p_i}\left(\int_{B_{\rho_1M_i^{\frac{p_i-1}{2}}}}M_i^{-\frac{p_i-1}{2}}|x|v_i^{p_i+1}+ \int_{B_{\rho_1M_i^{\frac{p_i-1}{2}}}}M_i^{-(p_i-1)}|x|^2v_i^{p_i+1}\right)\lesssim$$
 $$\lesssim M_i^{\frac{7-3p_i}{2}}M_i^{\frac{p_i^2-1}{2}\delta}\int_{B_{\rho_1M_i^{\frac{p_i-1}{2}}}}|x|(1+|x|)^{-2(p_i+1)}+$$
 $$+M_i^{4-2p_i}M_i^{\frac{p_i^2-1}{2}\delta}\int_{B_{\rho_1M_i^{\frac{p_i-1}{2}}}}|x|^2(1+|x|)^{-2(p_i+1)}\lesssim$$
 $$\lesssim M_i^{\frac{7-3p_i}{2}}M_i^{\frac{p_i^2-1}{2}\delta}M_i^{\frac{p_i-1}{2}\left[-2(p_i+1)+5\right]} +M_i^{4-2p_i}M_i^{\frac{p_i^2-1}{2}\delta}M_i^{\frac{p_i-1}{2}\left[-2(p_i+1)+6\right]}=$$
 $$=M_i^{\frac{7-3p_i}{2}}M_i^{\frac{p_i^2-1}{2}\delta}M_i^{\frac{p_i-1}{2}\left[-2(p_i+1)+5\right]} +M_i^{4-2p_i}M_i^{\frac{p_i^2-1}{2}\delta}M_i^{\frac{p_i-1}{2}\left[-2(p_i+1)+6\right]}=$$
 $$=M_i^{\frac{7-3p_i}{2} +5\frac{p_i-1}{2}}M_i^{\frac{p_i^2-1}{2}\delta}M_i^{-(p_i^2-1)} +M_i^{4-2p_i +6\frac{p_i-1}{2}}M_i^{\frac{p_i^2-1}{2}\delta}M_i^{-(p_i^2-1)}=$$
 $$=M_i^{p_i+1}M_i^{\frac{p_i^2-1}{2}\delta}M_i^{-(p_i^2-1)} = M_i^{-(p_i+1)(p_i-2)}M_i^{\frac{p_i^2-1}{2}\delta},$$
 
 $$\int_{B_{\rho_1}}\left(\frac{1}{4}R +\frac{1}{8}\Xi(R)\right)u_i^2 \lesssim M_i^2u_i(x_i)^{-2(p_i-1)}\int_{B_{\rho_1M_i^{\frac{p_i-1}{2}}}}v_i^2 \lesssim$$
 $$\lesssim M_i^{4-2p_i}M_i^{(p_i-1)\delta} \int_{B_{\rho_1M_i^{\frac{p_i-1}{2}}}}(1+|x|)^{-4}\lesssim M_i^{4-2p_i}M_i^{(p_i-1)\delta}\log M_i,$$
 and thanks to Lemma \ref{LemmaStimaPuntiIsolatiSemplici} it is obvious that
 $$\int_{\de B_{\rho_1}}\left(\left(\frac{1}{4}\frac{1}{p_i+1}\widetilde{R}u_i^{p_i+1}-\frac{1}{8}Ru_i^2  \right)\Xi\cdot\cerchio{\ni} +\ci{B}(x,u_i,\cerchio{\nabla}^Hu)\right)d\cerchio{\sigma}=$$
 $$=O(M_i^{(-2+\delta)(p_i-1)+2}).$$
 Finally by Proposition \ref{BlowUpIsolatiBolle} it holds that
 $$\int_{B_{\rho_1M_i^{\frac{p_i-1}{2}}}}u_i^{p_i+1} \gtrsim u_i(x_i)^{p_i+1}u_i(x_i)^{-2(p_i-1)}\cdot $$
 $$\cdot\int_{B_{R_i}}\frac{1}{u_i(x_i)^{p_i+1}}u_i\left(\delta_{M_i^{-\frac{p_i-1}{2}}}(x_i^{-1} \cdot x)\right)^{p_i+1}\gtrsim u_i(x_i)^{3-p_i} \gtrsim 1.$$
 So, using also Lemma \ref{StimaSublaplacianoCoordinate}, and the fact that $\frac{1}{p_i+1}-\frac{1}{4}= \frac{\tau_i}{4(p_i+1)}$, we get the thesis.
\end{proof}

\begin{lemma}\label{LemmaCSigma}
 In the hypotheses of Lemma \ref{LemmaStimaPuntiIsolatiSemplici}, if $\overline{x}$ is an isolated simple blow-up point then for every $\sigma\in(0,\frac{\overline{r}}{2})$
 $$\limsup_{i\to\infty}\max_{\de B_{\sigma}}M_iu_i(x) \le C(\sigma).$$
\end{lemma}

\begin{proof} 
 Thanks to Lemma \ref{Harnack}, it is sufficient to prove the thesis for $\sigma$ small enough. In particular, reasoning as in the proof of Proposition \ref{BlowUpIsolatiBolle}, we can suppose that $R>0$.
 
 Let $x_{\sigma}$ such that $d(x_{\sigma},x_i)=\sigma$, and define
 Let $w_i(x)=u_i(x_{\sigma})^{-1}u_i(x)$.
 Then $w_i$ satisfies
 \begin{equation}\label{EquazioneDimostrazioneLemma}
  L_{\theta}w_i = u_i(x_{\sigma})^{p_i-1}w_i^{p_i}.
 \end{equation}
 Thanks to Lemma \ref{LemmaStimaPuntiIsolatiSemplici} and to Harnack inequality, for every compact $K\subset B_{\rho_1}(\overline{x})\setminus\{\overline{x}\}$ there exists $C_K$ such that $C_K^{-1}\le w_i\le C_K$. Therefore, applying the regularity theory from Theorem \ref{TeoremaRegolarita}, we can deduce that up to subsequences $w_i\to w$ in $C^2_{\mathrm{loc}}(B_{\rho_1}(\overline{x})\setminus\{\overline{x}\})$, and since, by Lemma \ref{LemmaStimaPuntiIsolatiSemplici}, $u_i(x_{\sigma})\to 0$, passing to the limit equation \eqref{EquazioneDimostrazioneLemma} one gets that $L_{\theta}w=0$.
 
 Since the blow-up is isolated simple, and since Proposition \ref{BlowUpIsolatiBolle} implies that $r^{\frac{2}{p_i-1}}\overline{u}_i$ has a critical point in $(0,R_iM_i^{-\frac{p_i-1}{2}})$ after which it is decreasing, $r^{\frac{2}{p_i-1}}\overline{u}_i$ is decreasing in $(R_iM_i^{-\frac{p_i-1}{2}},\rho)$, and because
 $$u_i(x_{\sigma})^{-1}r^{\frac{2}{p_i-1}}\overline{u}_i(r)= r^{\frac{2}{p_i-1}}\overline{w}_i(r)\to r\overline{w}(r),$$
 $r\overline{w}(r)$ is decreasing on $(0,\rho)$.
 Since $w>0$, $w$ must be singular at $\overline{x}$.
 Corollary 9.1 in \cite{LZhu} can be extend to pseudohermitian geometry repeating the proof with minor adaptations; applying it we get that
 \begin{equation}\label{FormulaDimLemma1}
  -\int_{B_{\sigma}(x_i)}\Delta_bw_i = -\int_{\de B_{\sigma}(x_i)}\nabla^Hw_i\cdot\ni = -\int_{\de B_{\sigma}(\overline{x})}\nabla^Hw\cdot\ni + o(1) = c+o(1) > 0
 \end{equation}
 while integrating equation \eqref{EquazioneDimostrazioneLemma} we get
 \begin{equation}\label{FormulaDimLemma2}
  -4\int_{B_{\sigma}(x_i)}\Delta_bw_i = \int_{B_{\sigma}(x_i)}\left(-Rw_i + u_i(x_{\sigma})^{p_i-1}w_i^{p_i}\right) \le u_i(x_{\sigma})^{-1}\int_{B_{\sigma}(x_i)}u_i^{p_i}.
 \end{equation}
 But, calling $r_i= R_iM_i^{-\frac{p_i-1}{2}}$, by Lemmas \ref{BlowUpIsolatiBolle}, \ref{LemmaStimaPuntiIsolatiSemplici} and \ref{StimaTau} eventually
 $$\int_{B_{\sigma}(x_i)}u_i^{p_i} = \left(\int_{B_{r_i}(x_i)} +\int_{B_{\sigma}(x_i)\setminus B_{r_i}(x_i)}\right)u_i^{p_i} \lesssim$$
 $$\lesssim M^{-2(p_i-1)}M_i^{p_i}\int_{B_{R_i}(0)}(1+|x|)^{-2p_i}+ M_i^{-\lambda_ip_i}\int_{B_{\sigma}(x_i)\setminus B_{r_i}(x_i)}|x|^{(-2+\delta)p_i} \lesssim$$
 \begin{equation}\label{FormulaDimLemma3}
  \lesssim M_i^{2-p_i} + M_i^{-\lambda_ip_i}(R_iM_i^{-\frac{p_i-1}{2}})^{(-2+\delta)p_i +4} \lesssim M_i^{-1}.
 \end{equation}
 Putting together formulas \eqref{FormulaDimLemma1}, \eqref{FormulaDimLemma2} and \eqref{FormulaDimLemma3} we get that $u_i(x_i)u_i(x_{\sigma})$ is a bounded sequence. The thesis follows from Lemma \ref{Harnack}.
\end{proof}

\begin{proposizione}\label{LimiteBlowUp}
 If $\overline{x}$ is an isolated simple blow-up point then there exists $C$ such that
 $$M_iu_i(x)\le C d(x,x_i)^{-2}$$
 if $d(x,x_i)\le\frac{\rho}{2}$.
 Furthermore, up to subsequences, there exists $a>0$ such that
 $$M_iu(x)\to aG_{\overline{x}}(x) + b$$
 in $C^2_{\mathrm{loc}}(B_{\frac{\rho}{2}}(\overline{x})\setminus\{\overline{x}\})$, where $G_{\overline{x}}$ is the Green function of $L_{\theta}$ (which exists because $M$ has positive CR Yamabe class)
 and $L_{\theta}b=0$ on $B_{\frac{\rho}{2}}(\overline{x})$.
\end{proposizione}

\begin{proof}
 If this were not the case, then, up to subsequences, there would exist a sequence $\widetilde{x}_i$ with
 \begin{equation}\label{EquazioneDimLimiteBlowUp}
  M_iu_i(\widetilde{x}_i)d(x_i,\widetilde{x}_i)^2 \to\infty.
 \end{equation}
 Define $\widetilde{r}_i=d(x_i,\widetilde{x}_i)$.
 
 Since $\sup_{\lambda>0}\lambda^2U(\delta_{\lambda}(x))\le\frac{C}{|x|^2}$,
 it is easy to verify using
 Proposition \ref{BlowUpIsolatiBolle} and Lemma \ref{StimaTau} that $R_iM_i^{-\frac{p_i-1}{2}}\le\widetilde{r}_i\le\rho$.
 
 Define $\widetilde{u}_i=\widetilde{r}_i^{\frac{2}{p_i-1}}u_i\circ\delta_{\widetilde{r}_i}\circ L_{x_i}$ in $B_2$.
 $\widetilde{u}_i$ satisfies
 $$L_{\theta_i}\widetilde{u}_i = \widetilde{R}\widetilde{u}_i^{p_i}$$
 and verifies the hypotheses of Lemma \ref{LemmaCSigma}, therefore $\max_{\de B_1}\widetilde{u}_i(0)\widetilde{u}_i <\infty$. Using the definition of $\widetilde{u}_i$ and Lemma \ref{StimaTau}, this goes in contradiction with formula \eqref{EquazioneDimLimiteBlowUp}.
 
 Hence $M_iu_i$ is locally bounded in $B_{\frac{\rho}{2}}(\overline{x})\setminus\{\overline{x}\}$, and satisfies
 $$L_{\theta}(M_iu_i) = M_i^{1-p_i}(M_iu_i)^{p_i},$$
 therefore, applying regularity theory (see Theorem \ref{TeoremaRegolarita}),
 $$M_iu_i\to v \;\;\;\text{in}\;\;\; C^2_{\mathrm{loc}}(B_{\frac{\rho}{2}}(\overline{x})\setminus\{\overline{x}\})$$
 with $v$ satisfying $L_{\theta}v=0$.
 Known results about singular solutions (see for example Proposition 9.1 in \cite{LZhu}, which can adapted without difficulty to pseudohermitian geometry) imply the rest of the thesis, except the fact that $a>0$.
 This can be proved by proving that $v$ must be singular with the same proof of Lemma \ref{LemmaCSigma}.
\end{proof}

Using Proposition \ref{LimiteBlowUp} and Lemma \ref{StimaTau} we can improve the estimates of Lemma \ref{LemmaStimaPuntiIsolatiSemplici} and Lemma \ref{StimaSublaplacianoCoordinate} with the same proofs (which we do not repeat). We state only the estimates that we will need. The last statement of the following lemma can be proved by repeating the proof of Lemma \ref{StimaSublaplacianoCoordinate} using the estimates for CR normal coordinates (see Proposition A.5 in \cite{CMY1}).

\begin{lemma}\label{LemmaStime}
 If $\overline{x}$ is an isolated blow-up point then, with the notation of Lemma \ref{LemmaStimaPuntiIsolatiSemplici} there exists $\rho_1\in(0,\rho)$ such that
 $$u_i(x)\le C M_i^{-1}d(x,x_i)^{-2},$$
 $$|(u_i)_{,1}(x)|\le C M_i^{-1}d(x,x_i)^{-3},$$
 for $R_iM_i^{-\frac{p_i-1}{2}}\le d(x,x_i)\le\rho_1$,
 and in CR normal coordinates around $x_i$, for $\sigma\le\rho_1$,
 $$\left|\int_{B_{\sigma}(x_i)}\left(u_i +\Xi u_i\right)(\cerchio{\Delta}_bu_i -\Delta_b u_i)\right| \le C\sigma M_i^{-2}.$$
\end{lemma}

From now on we will focus on the case of constant $\widetilde{R}$, and by homogeneity we can suppose that $\widetilde{R}\equiv 1$. So the equation becomes
\begin{equation}\label{EquazioneRCostante}
 L_{\theta}u_i = u_i^{p_i}.
\end{equation}

\begin{osservazione}\label{OsservazioneCoordinateCR}
 In the following we will need to use CR normal coordinates, which involve a change of contact form.
 If $\phi$ is such that $\widetilde{\theta}=\phi^2\theta$ defines CR normal coordinates around some point, and $\widetilde{u}_i=\frac{u_i}{\phi}$, then, thanks to formula \eqref{LeggeTrasformazioneSublaplaciano}, $\widetilde{u}_i$ verifies
 $$L_{\widetilde{\theta}}\widetilde{u}_i = \phi^{-\tau_i}\widetilde{u}_i^{p_i}.$$
 It is easy to prove that the estimates proved until now for $u_i$, in particular the ones in Proposition \ref{BlowUpIsolatiBolle}, Proposition \ref{LimiteBlowUp} and Lemma \ref{LemmaStime}, hold for $\widetilde{u}_i$ too.
 It is also easy to prove that a point $\overline{x}$ is an isolated blow-up point for the sequence $u_i$ if and only if it is one for the sequence $\widetilde{u}_i$. We will say that $\overline{x}$ is an isolated blow-up point in CR normal coordinates if it is an isolated blow-up point for $\widetilde{u}_i$ with respect to the contact form $\widetilde{\theta}$.
 
 At some point we will need to consider CR coordinates around $x_i$ where $x_i\to\overline{x}$, defined by forms $\widetilde{\theta}_i=\phi_i^2\theta$. In such case $\widetilde{u}_i=\frac{u_i}{\phi_i}$ verifies the same estimates for analogous reasons because $\phi_i\to\phi$ smoothly on compact sets, because their construction involves only algebraic operations with objects depending smoothly on the base point (see \cite{JL3}).
\end{osservazione}

\begin{lemma}\label{SegnoTermineOrdineZero}
 If $\overline{x}$ is an isolated simple blow-up point for equation \eqref{EquazioneRCostante} and $M_iu_i\to h$ in $C^2_{\mathrm{loc}}(B_r(\overline{x})\setminus\{\overline{x}\})$, and in CR normal coordinates around $\overline{x}$ one has $h(x)=\frac{a}{|x|^2}+ A + o(1)$ with $a>0$, then $A\le 0$.
\end{lemma}

\begin{proof} 
 Using the notation introduced in Remark \ref{OsservazioneCoordinateCR} for CR normal coordinates around $\overline{x}$, given $\sigma\in(0,\rho_1)$ the function $\widetilde{u}_i\circ\delta_{\sigma}$ verifies the equation
 $$L_{\theta'_i}(\widetilde{u}_i\circ\delta_{\sigma}) = (\phi\circ\delta_{\sigma})^{-\tau_i}(\widetilde{u}_i\circ\delta_{\sigma})^{p_i}$$
 where $\theta'_i= \frac{1}{\sigma^2}\delta_{1/\sigma}^*\widetilde{\theta}_i$.
 
 Apply the Pohozaev identity to $\widetilde{u}_i\circ\delta_{\sigma}$ on $B_1(\overline{x})$:
 $$\int_{B_1(\overline{x})}\left(\left(\frac{1}{p_i+1}-\frac{1}{4}\right)(\phi\circ\delta_{\sigma})^{-\tau_i}(\widetilde{u}_i\circ\delta_{\sigma})^{p_i+1} -\frac{1}{4}(R\circ\delta_{\sigma})(\widetilde{u}_i\circ\delta_{\sigma})^2+ \right.$$
 $$+\frac{1}{4}\frac{1}{p_i+1}\Xi((\phi\circ\delta_{\sigma})^{-\tau_i})(\widetilde{u}_i\circ\delta_{\sigma})^{p_i+1}-\frac{1}{8}\Xi(R\circ\delta_{\sigma})(\widetilde{u}_i\circ\delta_{\sigma})^2+$$
 $$\left.-\left(\Xi (\widetilde{u}_i\circ\delta_{\sigma})+(\widetilde{u}_i\circ\delta_{\sigma})\right)(\Delta_b(\widetilde{u}_i\circ\delta_{\sigma})-\cerchio{\Delta}_b(\widetilde{u}_i\circ\delta_{\sigma}))\right)d\cerchio{V}=$$
 $$=\int_{\de B_1(\overline{x})}\left(\left(\frac{1}{4}\frac{1}{p_i+1}(\phi\circ\delta_{\sigma})^{-\tau_i}(\widetilde{u}_i\circ\delta_{\sigma})^{p_i+1}-\frac{1}{8}(R\circ\delta_{\sigma})(\widetilde{u}_i\circ\delta_{\sigma})^2  \right)\Xi\cdot\cerchio{\ni} +\right.$$
 \begin{equation}\label{PohozaevTeoremaOrdineZero}
  \left.+\ci{B}(x,\widetilde{u}_i\circ\delta_{\sigma},\cerchio{\nabla}^H(\widetilde{u}_i\circ\delta_{\sigma}))\right)d\cerchio{\sigma}
 \end{equation}
 Since $u_i\circ\delta_{\sigma}=M_iv_i\circ L_{x_1}\circ\delta_{\sigma M_i^{\frac{p_i-1}{2}}}$, $x_i\to\overline{x}$ and $R=O(|x|^2)$ and $R_{,1}=O(|x|)$ (see Appendix A.2 in \cite{CMY1}), and thanks to Lemma \ref{LemmaStime}, we have
 $$\int_{B_1(\overline{x})}(2R\circ\delta_{\sigma}+\Xi(R\circ\delta_{\sigma}))(u_i\circ\delta_{\sigma})^2d\cerchio{V} \lesssim$$
 $$\lesssim \sigma^{-4}M_i^{-2(p_i-1)}M_i^2M_i^{-(p_i-1)}\int_{B_{\sigma M_i^{\frac{p_i-1}{2}}}(0)}|x|^2(1+|x|)^{-4}d\cerchio{V}\lesssim \sigma^{-2}M_i^{-2},$$
$$\int_{\de B_1(\overline{x})}(\phi\circ\delta_{\sigma})^{-\tau_i}(\widetilde{u}_i\circ\delta_{\sigma})^{p_i+1}\Xi\cdot\cerchio{\ni} d\cerchio{\sigma} \lesssim M_i^{-(p_i-1)}\sigma^{-2(p_i-1)},$$
and since $\frac{4}{p_i+1}-1 = \frac{\tau_i}{p_i+1}$,
$$\int_{B_1(\overline{x})}\left[\left(\frac{4}{p_i+1}-1\right)(\phi\circ\delta_{\sigma})^{-\tau_i}(\widetilde{u}_i\circ\delta_{\sigma})^{p_i+1}+ \frac{1}{p_i+1}\Xi((\phi\circ\delta_{\sigma})^{-\tau_i})(\widetilde{u}_i\circ\delta_{\sigma})^{p_i+1}\right]=$$
$$=\frac{\tau_i}{p_i+1}\int_{B_1(\overline{x})}(\phi\circ\delta_{\sigma})^{-\tau_i-1}\left[\phi\circ\delta_{\sigma}+ \Xi(\phi\circ\delta_{\sigma})\right](\widetilde{u}_i\circ\delta_{\sigma})^{p_i+1}\ge 0$$
if $\sigma$ is small enough.

Multiplying formula \eqref{PohozaevTeoremaOrdineZero} by $\sigma^3M_i^2$, taking $\lim_{\sigma\to 0}\limsup_{i\to\infty}$ and using the above estimates and the last statement in Lemma \ref{LemmaStime}, one gets that
$$\lim_{\sigma\to 0}\limsup_{i\to\infty}\sigma^3M_i^2\int_{\de B_1(\overline{x})}\ci{B}(x,\widetilde{u}_i\circ\delta_{\sigma},\cerchio{\nabla}^H(\widetilde{u}_i\circ\delta_{\sigma}))d\cerchio{\sigma} \ge 0.$$
But by Proposition \ref{LimiteB}
$$\lim_{\sigma\to 0}\limsup_{i\to\infty}\sigma^3M_i^2\int_{\de B_1(\overline{x})}\ci{B}(x,\widetilde{u}_i\circ\delta_{\sigma},\cerchio{\nabla}^H(\widetilde{u}_i\circ\delta_{\sigma}))d\cerchio{\sigma} =$$
$$=\lim_{\sigma\to 0}\sigma^3\int_{\de B_1(\overline{x})}\ci{B}(x,h\circ\delta_{\sigma},\cerchio{\nabla}^H(h\circ\delta_{\sigma}))d\cerchio{\sigma} = -\frac{c_na}{\gamma_n}A$$
and therefore $A\le 0$.
\end{proof}

\begin{proposizione}\label{BlowUpIsolatiSonoSemplici}
 Isolated blow-up points for equation \eqref{EquazioneRCostante} are isolated simple blow-up points in CR normal coordinates.
\end{proposizione}

\begin{proof}
 If this is not the case, up to subsequences there exists $\overline{\mi}_i\to 0$ such that $r^{\frac{2}{p_i-1}}\overline{u}_i(r)$ has at least two critical points in $(0,\overline{\mi}_i)$.
 Up to further passage to subsequences, one can assume that the thesis of Proposition \ref{BlowUpIsolatiBolle} holds.
 Then
 $r^{\frac{2}{p_i-1}}\overline{\widetilde{u}}_i(r)$
 has exactly one critical point in $(0,r_i)$ where $r_i=R_iM_i^{-\frac{p_i-1}{2}}$.
 
 Since critical points form a closed set, let $\mi_i\in(r_i,\overline{\mi}_i)$ be the second critical point.
 
 We pass in CR normal coordinates around $x_i$ using the notation from Remark \ref{OsservazioneCoordinateCR}, and we define
 $\xi_i=\mi_i^{\frac{2}{p_i-1}}\widetilde{u}_i\circ\delta_{\mi_i}$
 on $B_{1/\mi_i}(\overline{x})$.
 Then, following Remark \ref{OsservazioneRiscalamento}, $\xi_i$ satisfies the equation
 $$L_{\theta_i}\xi_i= (\phi_i\circ\delta_{\mi_i})^{-\tau_i}\xi_i^{p_i}$$
 where
 $\theta_i=\mi_i^2\left(\delta_{\mi_i}\right)^*\widetilde{\theta}_i$.
 
 Furthermore $\xi_i(x)\le Cd(x,x_i)^{-\frac{2}{p_i-1}}$ by definition of isolated blow-up point, 
 $$\xi_i(0) = \mi_i^{\frac{p_i-1}{2}}\widetilde{u}_i(x_i) \ge R_i^{\frac{p_i-1}{2}}M_i^{-1}\widetilde{u}_i(x_i)\to\infty,$$
 $r^{\frac{2}{p_i-1}}\overline{\xi}_i(r)$ has exactly one critical point on $(0,1)$,
 \begin{equation}\label{FormulaDimBlowUpIsolati}
  \left.\frac{d}{dr}\right|_{r=1}\left(r^{\frac{2}{p_i-1}}\xi_i(0)\overline{\xi}_i(r)\right)=0,
 \end{equation}
 and the domains of $\xi_i$ exaust $\H^1$.
 
 Proposition \ref{LimiteBlowUp} and an application of Lemma \ref{Harnack} to increasing annuli permit to prove that on every compact $K\subset\H^1\setminus\{0\}$ there exists $C_K$ such that
 $$C_K^{-1} \le \xi(0)\xi\le C_K$$
 on K. This and Proposition \ref{LimiteBlowUp} allow to conclude that, up to subsequences,
 $$\xi_i(0)\xi(x) \to h(x)= \frac{a}{|x|^2} +b(x)$$
 on compact subsets of $\H^1\setminus\{0\}$, with $a>0$ and $b$ satisfying $\Delta_bb=0$ on $\H^1$.
 Since $h$ is positive, then $\liminf_{x\to\infty}b(x)\ge 0$, and thus by the maximum principle $b\ge 0$, and therefore by the Liouville theorem on $\H^n$ (which, for example, can be derived by the Harnack inequality) $b(x)\equiv b$ is a constant.
 
 Furthermore,
 passing to the limit in equation \eqref{FormulaDimBlowUpIsolati}
 allows to deduce that
 $$0=\left.\frac{d}{dr}\right|_{r=1}(r\overline{h}(r))=C\left.\frac{d}{dr}\right|_{r=1}\left(\frac{a}{r}+ br \right) = (-a+b)C$$
 (where $C=\int_{\de B_1(0)}d\cerchio{\sigma}$) therefore $b=a>0$. But this goes against Lemma \ref{SegnoTermineOrdineZero} (here we are in a situation slightly different from the hypotheses of that lemma because the contact form $\widetilde{\theta}_i$ is not fixed, but the proof can be adapted straightforwardly).
\end{proof}

\begin{proposizione}\label{ElencoPropBlowUp}
 For every $\e>0$, $R>1$, there exist $C_0$ e $C_1$ such that if
 $u$ solves equation \eqref{EquazioneRCostante}
 with $\max u>C_0$, then there exist $x_1,\ldots,x_k$ verifying:
 \begin{enumerate}
  \item $0\le\tau_i<\e$;
  \item the $x_i$ are local maxima of $u$, and calling $\overline{r}_j=Ru(x_j)^{-(p-1)/2}$, the balls $B_{\overline{r}_j}(x_j)$ are disjoint and in pseudohermitian coordinates around $x_i$
  $$\N{\frac{1}{u(x_j)}u\left(\delta_{M_i^{-\frac{p_i-1}{2}}}(x_i^{-1} \cdot x)\right) -U(x)}_{C^2(B_R(0))}<\e;$$
  \item $d(x,\{x_1,\ldots,x_k\})^{\frac{2}{p-1}}\le C_1$ per ogni $x\in M$.
 \end{enumerate}
\end{proposizione}

\begin{proof}
 The proof is equal to the analogous statement for Riemannian manifolds, but using Theorem \ref{Classificazione} instead of the corresponding result in $\R^n$ by Caffarelli, Gidas and Spruck (Corollary 8.2 in \cite{CGS}) and Ma-Ou's theorem (Theorem 1.1 in \cite{MO}) instead of the corresponding result in $\R^n$ by Gidas and Spruck (Theorem 1.1 in \cite{GS}). See for example Proposition 5.1 in \cite{LZhu}.
\end{proof}

\begin{proposizione}\label{BlowUpSonoIsolati}
 In the same hypotheses of the Proposition \ref{ElencoPropBlowUp}, there exists $\delta^*$
 depending only by
 $(M,\ci{H},\theta)$, $R$ and $\e$
 such that, in the notation of the same proposition, if $\max u>C_0$ then $d(x_j,x_{\ell})\ge\delta^*$ for every $j\ne\ell$.
\end{proposizione}

\begin{proof} 
 Were the thesis false, there would exist $\e$, $R$ and sequences $p_i$ and $u_i$ of solutions of
 equation \eqref{EquazioneRCostante} with local maxima $x_1^i,\ldots,x_{N_i}^i$ such that (up to relabeling)
 $$\sigma_i = d(x_1^i,x_2^i)=\min_{j,\ell}d(x_j^i,x_{\ell}^i)\to 0.$$
 By point 1 of Proposition \ref{ElencoPropBlowUp}, $p_i\to 3$.
 Since by point 2 in the same proposition $B_{Ru_i(x_1^i)}\cap B_{Ru_i(x_2^i)}=\emptyset$, then $u_i(x_1^i),u_i(x_2^i)\to\infty$.
 
 Using CR normal coordinates around $x_1^i$ with the notation from Remark \ref{OsservazioneCoordinateCR}, we have
 $$L_{\widetilde{\theta}_i}\widetilde{u}_i = \phi_i^{-\tau_i}\widetilde{R}\widetilde{u}_i^{p_i}.$$
 Let $w_i=\sigma_i^{\frac{2}{p_i-1}}\widetilde{u}_i\circ\delta_{\sigma_i}$ defined on $B_{1/\sigma}\subset\H^1$.
 Then $w_i$ satisfies
 $$L_{\theta'_i}w_i = (\phi_i\circ\delta_{\sigma_i})^{-\tau_i}\widetilde{R}w_i^{p_i}$$
 (where $\widetilde{\theta}_i = \sigma_i^2\left(\delta_{\sigma_i}\right)^*\widetilde{\theta}$).
 
 Calling $\widetilde{x}_i\in\H^1$ the point corresponding to $x_2^i$ in CR coordinates, since the Carnot-Carathéodory distance of $M$ converges to the one of $\H^1$, $d(\widetilde{x}_i,0)\to 1$, and therefore up to subsequences $\widetilde{x}_i\to \widetilde{x}$ with $d(\widetilde{x},0)=1$.
 
 Let $S\subset\H^1$ be the set of blow-up points for $w_i$. By construction
 $$\min_{x,y\in S,x\ne y}d(x,y)\ge 1,$$
 and thus, using point 3 in Proposition \ref{ElencoPropBlowUp}, they are isolated blow up points.
 It is easy to see that if $w_i(0)\to\infty$ then $0\in S$, and similarly for $\widetilde{x}$.
 
 Since, by point 2 in Proposition \ref{ElencoPropBlowUp},
 $$\sigma_i\ge\max\left\{Ru_i(x_i^1)^{-\frac{p_i-1}{2}}, Ru_i(x_i^2)^{-\frac{p_i-1}{2}}\right\},$$
 then
 \begin{equation}\label{DisuguaglianzaDimPuntiIsolati}
  w_i(0),w_i(\widetilde{x}_i)\ge R^{\frac{2}{p_i-1}}.
 \end{equation}
 
 We want to prove that up to subsequences $w_i(0),w_i(\widetilde{x}_i)\to\infty$.
 
 Suppose that $w_i(0)\to\infty$ but $w_i(\widetilde{x}_i)$ stays bounded.
 Then, by Proposition \ref{BlowUpIsolatiSonoSemplici}, $0$ would be an isolated simple blow-up point, and using Proposition \ref{LimiteBlowUp}, and the maximum principle and Harnack inequality in a way similar to the proof of Proposition \ref{BlowUpIsolatiBolle}, it could be proved that $w_i(\widetilde{x}_i)\to 0$, against formula \eqref{DisuguaglianzaDimPuntiIsolati}.
 If $w_i(\widetilde{x}_i)\to\infty$ but $w_i(0)$ stays bounded, the reasoning is analogous.
 
 Suppose that $w_i(0)$ and $w_i(\widetilde{x}_i)$ both stay bounded.
 Reasoning as in the other cases, if $S\ne\emptyset$ then it can be proved that $w_i(0),w_i(\widetilde{x}_i)\to 0$, against formula \eqref{DisuguaglianzaDimPuntiIsolati}.
 Hence $S=\emptyset$, and therefore $w_i$ is locally bounded.
 Therefore, using regularity theory as in the proof of Proposition \ref{BlowUpIsolatiBolle}, it can be proved that up to subsequences $w_i\to v$ in $C^2_{\mathrm{loc}}$, with $v$ satisfying $-4\Delta_bv = v^3$, $v\ge 0$ and $\nabla^Hv(0)=\nabla^Hv(\widetilde{x})=0$. By Theorem \ref{Classificazione}
 $v\equiv 0$. This goes yet again against formula \eqref{DisuguaglianzaDimPuntiIsolati}.
 
 So $0,\widetilde{x}\in S$. Now, arguing as before with the maximum principle and Harnack inequality, it can be proved that $w_i(0)w_i$ is locally bounded in $\H^1\setminus S$, and thus, up to subsequences, it converges in $C^2_{\mathrm{loc}}(\H^1\setminus S)$ to a function $h$, and because of Proposition \ref{LimiteBlowUp}
 $$h(x)= \frac{a}{|x|^2} + \frac{b}{|\widetilde{x}^{-1}x|^2} + c(x)$$
 where $a,b>0$, $c$ is smooth in $\H^1\setminus(S\setminus\{0,\widetilde{x}\})$ and verifies $\Delta_bc=0$.
 Since $h>0$, by the maximum principle $c\ge 0$.
 Therefore
 $$h(x)= \frac{a}{|x|^2} + A + o(1)$$
 with $A>0$, but, arguing as in the proof of Proposition \ref{BlowUpIsolatiSonoSemplici}, this is in contradiction with Proposition \ref{SegnoTermineOrdineZero}.
\end{proof}


Now we can finally prove the main result of this work.

\begin{proof}[Proof of Theorem \ref{IlTeorema}]
 Using subelliptic estimates and the Harnack inequality as in the proof of Proposition \ref{BlowUpIsolatiBolle}, it is sufficient to prove that there exists no blow-up point.
 
 If this is not the case, by point 1 of Proposition \ref{ElencoPropBlowUp} we know that $p_i\to 3$, and by Proposition \ref{BlowUpSonoIsolati} and point 3 of Proposition \ref{ElencoPropBlowUp} there exist a finite number of isolated blow-up points $x^1_i\to\overline{x}^1,\ldots,x_i^N\to\overline{x}^N$, which, by Proposition \ref{BlowUpIsolatiSonoSemplici}, are isolated simple in CR normal coordinates.
 
 Up to subsequences and renaming of the blow-up points, we can suppose that $u_i(x_i^1)\le u_i(x_i^k)$ for every $i$ and $k$.
 Define $w_i=u_i(x_i^1)u_i$.
 By Proposition \ref{LimiteBlowUp},
 \begin{equation}\label{StimaDimostrazioneFinale}
  w_i(x)\le Cd(x,x_i^k) \;\;\text{in}\;\; B_{\frac{\rho}{2}}(x_i^k)
 \end{equation}
 By point 3 in Proposition \ref{ElencoPropBlowUp}, $u_i$ is uniformly bounded in $M\setminus\cup_kB_{\frac{\rho}{2}}(x^k)$. Since the CR manifold has positive CR Yamabe class, there exists a metric of positive Webster curvature (this follows either by the solution of the CR Yamabe problem, or more elementarily, thanks to formula \eqref{LeggeTrasformazioneCurvatura}, by taking $\phi_1^2\theta$ where $\phi_1$ is the first eigenfunction of $L_{\theta}$).
 Then, arguing as in the proof of Proposition \ref{BlowUpIsolatiBolle}, it can be proved that
 $$\max_{M\setminus\cup_kB_{\frac{\rho}{2}}(x^k)}u_i \le C \min_{\de\left(M\setminus\cup_kB_{\frac{\rho}{2}}(x^k)\right)}u_i $$
 which along with \eqref{StimaDimostrazioneFinale} shows that $w_i$ is uniformly bounded in $M\setminus\cup_kB_{\frac{\rho}{2}}(x^k)$.
 This, inequality \eqref{StimaDimostrazioneFinale}, Theorem \ref{TeoremaRegolarita} and Proposition \ref{LimiteBlowUp} imply that, up to subsequences,
 $$w_i(x) = u_i(x_i^1)u_i(x)\to h(x)=\sum_{k=1}^N a_kG_{x^k}(x) +b(x)$$
 where $a_k\ge 0$ are constants, $a_1>0$ and $b$ is a function satisfying $L_{\theta}b=0$.
 Since $M$ has positive CR Yamabe class, $L_{\theta}$ has trivial kernel, therefore $b=0$.
 
 By Proposition 5.3 in \cite{CMY1}, thanks to hypothesis that $m_{x^1}>0$, in CR normal coordinates around $x^1$ we have
 $$ G_{x^1}(x)= \frac{1}{4\pi|x|^2} + A + w(x)$$
 with $A>0$, $w(0)=0$, $w(x)=o(1)$.
 Thus
 $$h(x)= \frac{a_1}{4\pi|x|^2} + A' + o(1)$$
 for $x\to 0$, with $A'>0$. But this goes against Lemma \ref{SegnoTermineOrdineZero}.
\end{proof}

\appendix
\section{Some computations in pseudohermitian normal coordinates}
Let $M$ be a three-dimensional pseudoconvex pseudohermitian manifold and $x\in M$.

\begin{proposizione}
 In pseudohermitian normal coordinates
 $$\theta^1 = (1+O(|x|^2))\cerchio{\theta}^1 + O(|x|^2)\cerchio{\theta}^{\con{1}} + O(|x|)\cerchio{\theta}$$
 $$\theta = (1+O(|x|^2))\cerchio{\theta} + O(|x|^3)\theta^1 + O(|x|^3)\theta^{\con{1}}.$$
 \begin{equation}\label{omega}
  {\omega_1}^1= O(|x|)\cerchio{\theta}^1 + O(|x|)\cerchio{\theta}^{\con{1}} + O(|x|)\cerchio{\theta}
 \end{equation}
\end{proposizione}

\begin{proof}
 It follows from Proposition 2.5 in \cite{JL3}.
\end{proof}

\begin{lemma}\label{LemmaRelazioniPH}
 In pseudohermitian normal coordinates
 \begin{equation}\label{RelazioniCoordinatePH}
  \begin{cases}
   Z_1= (1+O(|x|^2))\cerchio{Z}_1 +O(|x|^2)\cerchio{Z}_{\con{1}}+ O(|x|^3)\cerchio{T}\\
   Z_{\con{1}}= O(|x|^2)\cerchio{Z}_1 +(1+O(|x|^2))\cerchio{Z}_{\con{1}}+ O(|x|^3)\cerchio{T}\\
   T = O(|x|)\cerchio{Z}_1 + O(|x|)\cerchio{Z}_{\con{1}} +(1+ O(|x|^2))\cerchio{T}
  \end{cases},
 \end{equation}
 and
 \begin{equation}\label{RelazioniInverseCoordinatePH}
  \begin{cases}
   \cerchio{Z}_1= (1+O(|x|^2))Z_1 +O(|x|^2)Z_{\con{1}}+ O(|x|^3)T\\
   \cerchio{Z}_{\con{1}} = O(|x|^2)Z_1 + (1+ O(|x|^2))Z_{\con{1}} + O(|x|^3)T\\
   \cerchio{T} = O(|x|)Z_1 + O(|x|)Z_{\con{1}} + (1+O(|x|^2))T
  \end{cases}
 \end{equation}
\end{lemma}

\begin{proof}
 Letting $Z_1= a\cerchio{Z}_1+ b\cerchio{Z}_{\con{1}} + c\cerchio{T}$ and applying $\theta^1$, $\theta^{\con{1}}$ and $\theta$, we get
 $$\begin{cases}
    1 = (1+O(|x|^2))a + O(|x|^2)b +O(|x|)c,\\
    0 = (1+O(|x|^2))b + O(|x|^2)a + O(|x|)c,\\
    0= (1+O(|x|^2))c + O(|x|^3)a + O(|x|^3)b
   \end{cases}$$
 respectively.
 The third one implies that $c=O(|x|^3)$, and using this in the other two allows to deduce that $a=1+O(|x|^2)$ and $b=O(|x|^2)$.
 Therefore
 $$Z_1= (1+O(|x|^2))\cerchio{Z}_1 +O(|x|^2)\cerchio{Z}_{\con{1}}+ O(|x|^3)\cerchio{T}.$$
 Letting $T= d\cerchio{Z}_1+ e\cerchio{Z}_{\con{1}} + f\cerchio{T}$ and applying $\theta^1$, $\theta^{\con{1}}$ and $\theta$, we get
 $$\begin{cases}
    0 = (1+O(|x|^2))d + O(|x|^2)e + O(|x|)f,\\
    0 = O(|x|^2)d + (1+O(|x|^2))e + O(|x|)f\\
    1 = (1+O(|x|^2))f + O(|x|^3)d + O(|x|^3)e.
   \end{cases}$$
 Arguing as before,
 these imply that $d=O(|x|)$, $e=O(|x|)$ and $f= 1+ O(|x|^2)$.
 Therefore
 $$T = O(|x|)\cerchio{Z}_1 + O(|x|)\cerchio{Z}_{\con{1}} +(1+ O(|x|^2))\cerchio{T}.$$
 So we proved the first part of the Lemma.
 
 We can write the formulas we proved as
 \begin{equation*}
  \left(
   \begin{array}{c}
    Z_1\\
    Z_{\con{1}}\\
    T
   \end{array}
  \right)
  =
  \left(
   I+
   \left(
    \begin{array}{ccc}
     O(|x|^2) & O(|x|^2) & O(|x|^3) \\
     O(|x|^2) & O(|x|^2) & O(|x|^3) \\
     O(|x|) & O(|x|) & O(|x|^2)
    \end{array}
   \right)
  \right)
  \left(
   \begin{array}{c}
    \cerchio{Z}_1\\
    \cerchio{Z}_{\con{1}}\\
    \cerchio{T}
   \end{array}
  \right).
 \end{equation*}
 Applying the Taylor expansion for the matrix inverse $(I+A)^{-1}=I-A+A^2 +O(\N{A}^3)$, we get
 \begin{equation*}
  \left(
   \begin{array}{c}
    \cerchio{Z}_1\\
    \cerchio{Z}_{\con{1}}\\
    \cerchio{T}
   \end{array}
  \right)
  =
  \left(
   I+
   \left(
    \begin{array}{ccc}
     O(|x|^2) & O(|x|^2) & O(|x|^3) \\
     O(|x|^2) & O(|x|^2) & O(|x|^3) \\
     O(|x|) & O(|x|) & O(|x|^2)
    \end{array}
   \right)
   +\right.
 \end{equation*}
 \begin{equation*}
  +\left.
   \left(
    \begin{array}{ccc}
     O(|x|^4) & O(|x|^4) & O(|x|^5) \\
     O(|x|^4) & O(|x|^4) & O(|x|^5) \\
     O(|x|^3) & O(|x|^3) & O(|x|^4)
    \end{array}
   \right)
   +O(|x|^3)
  \right)
  \left(
   \begin{array}{c}
    Z_1\\
    Z_{\con{1}}\\
    T
   \end{array}
  \right)=
 \end{equation*} 
 \begin{equation*}
  =
  \left(
   I+
   \left(
    \begin{array}{ccc}
     O(|x|^2) & O(|x|^2) & O(|x|^3) \\
     O(|x|^2) & O(|x|^2) & O(|x|^3) \\
     O(|x|) & O(|x|) & O(|x|^2)
    \end{array}
   \right)
  \right)
  \left(
   \begin{array}{c}
    Z_1\\
    Z_{\con{1}}\\
    T
   \end{array}
  \right).
 \end{equation*}
 This implies the second part of the thesis.
\end{proof}

\begin{lemma}
 In pseudohermitian normal coordinates
 \begin{equation}\label{RelazioniInverseCoordinatePH2}
  \begin{cases}
   \cerchio{Z}_1^2 = (1+ O(|x|^2))Z_1^2 + O(|x|)Z_1 + O(|x|)Z_{\con{1}} + O(|x|^2)T+\\
   +O(|x|^2)Z_{\con{1}}Z_1+ O(|x|^3)Z_1T + O(|x|^4)Z_{\con{1}}^2 + O(|x|^5)Z_{\con{1}}T + O(|x|^6)T^2\\
   \cerchio{Z}_{\con{1}}\cerchio{Z}_1 = (1+ O(|x|^2))Z_{\con{1}}Z_1 + O(|x|)Z_1 + O(|x|)Z_{\con{1}} + O(|x|^2)T+\\
   + O(|x|^2)Z_1^2+ O(|x|^2)Z_{\con{1}}^2 + O(|x|^3)Z_1T + O(|x|^3)Z_{\con{1}}T + O(|x|^6)T^2\\
   \cerchio{Z}_1\cerchio{T} = (1+O(|x|^2))Z_1T + O(1)Z_1 + O(1)Z_{\con{1}} + O(|x|)T+\\
   + O(|x|)Z_1^2 + O(|x|)Z_{\con{1}}Z_1 + O(|x|^2)Z_{\con{1}}T + O(|x|^3)T^2 + O(|x|^3)Z_{\con{1}}^2\\
   \cerchio{T}^2 = (1+O(|x|^2))T^2+ O(1)Z_1 + O(1)Z_{\con{1}} + O(|x|)T + O(|x|)Z_1T +\\
   +O(|x|)Z_{\con{1}}T + O(|x|^2)Z_{\con{1}}Z_1 +O(|x|^2)Z_1^2 + O(|x|^2)Z_{\con{1}}^2
  \end{cases}
 \end{equation}
\end{lemma}

\begin{proof}
 It follows from Lemma \ref{LemmaRelazioniPH} and some computations.
\end{proof}

\begin{lemma}\label{StimeDerivateCoordinateCR}
 For any function $f$ in CR normal coordinates around $\overline{x}$
 $$|f_{,1}- \cerchio{Z}_1f|\lesssim (|f_{,1}|+f_{,\con{1}}|)|x|^2 + |f_{,0}||x|^3,$$
 $$|f_{,0}- \cerchio{T}f|\lesssim (|f_{,1}|+f_{,\con{1}}|)|x| + |f_{,0}||x|^2,$$
 $$|f_{,11}- \cerchio{Z}_1^2f|\lesssim (|f_{,1}|+|f_{,\con{1}}|)|x| + (|f_{,0}|+|f_{,11}|+|f_{,1\con{1}}|)|x|^2 + |f_{,01}||x|^3+$$
 $$+|f_{,\con{1}\con{1}}||x|^4 +|f_{,0\con{1}}||x|^5 + |f_{,00}||x|^6,$$
 $$|f_{,1\con{1}}- \cerchio{Z}_{\con{1}}\cerchio{Z}_1f|\lesssim (|f_{,1}|+f_{,\con{1}}|)|x| + (|f_{,0}|+|f_{,11}|+|f_{,1\con{1}}|+|f_{,\con{1}\con{1}}|)|x|^2+$$
 $$+ (|f_{,01}|+|f_{,0\con{1}}|)|x|^3 + |f_{,00}||x|^6,$$
 $$|f_{,01}- \cerchio{Z}_1\cerchio{T}f|\lesssim |f_{,1}|+|f_{,\con{1}}| + (|f_{,0}|+|f_{,11}|+|f_{,1\con{1}}|)|x| +(|f_{,01}| +|f_{,0\con{1}}|)|x|^2 +$$
 $$+ (|f_{,00}|+ |f_{,\con{1}\con{1}}|)|x|^3,$$
 $$|f_{,00}- \cerchio{T}^2f|\lesssim |f_{,1}|+|f_{,\con{1}}| + (|f_{,0}|+|f_{,01}|+|f_{,0\con{1}}|)|x| + (|f_{,11}|+|f_{,1\con{1}}| + |f_{,\con{1}\con{1}}| +|f_{,00}|)|x|^2.$$
\end{lemma}

\begin{proof}
 The first two estimates follow from formulas \eqref{RelazioniInverseCoordinatePH} and \eqref{omega}. The other ones from formulas \eqref{RelazioniInverseCoordinatePH2} and \eqref{omega}.
\end{proof}

\begin{lemma}\label{SublaplacianoCoordinate}
 In pseudohermitian normal coordinates around a point $x$
 $$\Delta_b = \cerchio{\Delta}_b +O(|x|)\cerchio{Z}_1+O(|x|)\cerchio{Z}_{\con{1}} +O(|x|^2)\cerchio{T} + O(|x|^2)\cerchio{Z}_{\con{1}}\cerchio{Z}_1 +O(|x|^2)\cerchio{Z}_1\cerchio{Z}_1 +$$
 $$+O(|x|^2)\cerchio{Z}_{\con{1}}\cerchio{Z}_{\con{1}}+ O(|x|^3)\cerchio{T}\cerchio{Z}_1+ O(|x|^3)\cerchio{T}\cerchio{Z}_{\con{1}} +O(|x|^6)\cerchio{T}^2$$
\end{lemma}

\begin{proof}
 Using formulas \eqref{RelazioniCoordinatePH} we get
 $$\Delta_b  = Z_1Z_{\con{1}}+ Z_{\con{1}}Z_{1} - {\omega_{1}}^{1}(Z_{\con{1}})Z_{1} - {\omega_{\con{1}}}^{\con{1}}(Z_1)Z_{\con{1}}=$$
 $$= \left((1+O(|x|^2))\cerchio{Z}_1 +O(|x|^2)\cerchio{Z}_{\con{1}} + O(|x|^3)\cerchio{T}\right)\cdot$$
 $$\cdot\left(O(|x|^2)\cerchio{Z}_1 +(1+O(|x|^2))\cerchio{Z}_{\con{1}}+ O(|x|^3)\cerchio{T}\right) +$$
 $$- \left(O(|x|)\cerchio{\theta}^1 + O(|x|)_{\con{1}}\cerchio{\theta}^{\con{1}} + O(|x|)\cerchio{\theta}\right)\cdot$$
 $$\cdot\left(O(|x|^2)\cerchio{Z}_1 +(1+O(|x|^2))\cerchio{Z}_{\con{1}}+ O(|x|^3)\cerchio{T}\right)\cdot$$
 $$\cdot \left((1+O(|x|^2))\cerchio{Z}_1 +O(|x|^2)\cerchio{Z}_{\con{1}}+ O(|x|^3)\cerchio{T}\right) +c.c.=$$
 $$= \cerchio{\Delta}_b  +O(|x|)\cerchio{Z}_1+O(|x|)\cerchio{Z}_{\con{1}} +O(|x|^2)\cerchio{T} + O(|x|^2)\cerchio{Z}_{\con{1}}\cerchio{Z}_1 +O(|x|^2)\cerchio{Z}_1\cerchio{Z}_1 +$$
 $$+O(|x|^2)\cerchio{Z}_{\con{1}}\cerchio{Z}_{\con{1}}+ O(|x|^3)\cerchio{T}\cerchio{Z}_1+ O(|x|^3)\cerchio{T}\cerchio{Z}_{\con{1}} +O(|x|^6)\cerchio{T}^2$$
 where ``$+c.c.$'' means that the complex conjugate of the previous terms is to be added.
\end{proof}

\begin{lemma}\label{DisuguaglianzaSublaplaciano}
 For any function $f$ in CR normal coordinates around $\overline{x}$
 $$|\Delta_bf - \cerchio{\Delta}_bf| \lesssim |x||f_{,1}|+ |x||f_{,\con{1}}| + |x|^2|f_{,0}| + |x|^2|f_{,1\con{1}}| +|x|^2|f_{,11}| +$$
 $$+|x|^2|f_{,\con{1}\con{1}}|+ |x|^3|f_{01}| + |x|^3|f_{,0\con{1}}| +|x|^6|f_{,00}|$$
\end{lemma}

\begin{proof}
 By Lemma \ref{SublaplacianoCoordinate}
 $$|\Delta_bf - \cerchio{\Delta}_bf| \lesssim |x||\cerchio{Z}_1f|+ |x||\cerchio{Z}_{\con{1}}f| + |x|^2|\cerchio{T}f| + |x|^2|\cerchio{Z}_{\con{1}}\cerchio{Z}_1f| +|x|^2|\cerchio{Z}_1\cerchio{Z}_1f| +$$
 $$+|x|^2|\cerchio{Z}_{\con{1}}\cerchio{Z}_{\con{1}}f|+ |x|^3|\cerchio{T}\cerchio{Z}_1f| + |x|^3|\cerchio{T}\cerchio{Z}_{\con{1}}f| +|x|^6|\cerchio{T}^2f|$$
 and the thesis follows from Lemma \ref{StimeDerivateCoordinateCR} and some computations.
\end{proof}

 \textsc{Claudio Afeltra, Department of Mathematics, University of Trento, Via Sommarive 14, 38123 Povo (Trento), Italy}
 
 \textit{Email address}:  \texttt{claudio.afeltra@unitn.it}
 
\end{document}